\documentclass[12pt,english]{amsart}
\usepackage{lmodern}
\usepackage{helvet}

\usepackage[T1]{fontenc}
\usepackage[latin9]{inputenc}
\usepackage{amsthm}
\usepackage{amstext}
\usepackage{amssymb}
\usepackage{graphicx}
\usepackage{setspace}
\usepackage{esint}

\makeatletter

\newcommand{\noun}[1]{\textsc{#1}}

\numberwithin{equation}{section}
\numberwithin{figure}{section}
\numberwithin{table}{section}
\theoremstyle{plain}
\newtheorem{thm}{\protect\theoremname}[section]
  \theoremstyle{plain}
  \newtheorem{prop}[thm]{\protect\propositionname}
  \theoremstyle{remark}
  \newtheorem{rem}[thm]{\protect\remarkname}
  \theoremstyle{plain}
  \newtheorem{conjecture}[thm]{\protect\conjecturename}
  \theoremstyle{definition}
  \newtheorem{defn}[thm]{\protect\definitionname}
  \theoremstyle{definition}
  \newtheorem{example}[thm]{\protect\examplename}
\newenvironment{lyxlist}[1]
{\begin{list}{}
{\settowidth{\labelwidth}{#1}
 \setlength{\leftmargin}{\labelwidth}
 \addtolength{\leftmargin}{\labelsep}
 }}
{\end{list}}
  \theoremstyle{plain}
  \newtheorem{lem}[thm]{\protect\lemmaname}

\usepackage{amsfonts}
\usepackage{amssymb}
\usepackage[all]{xy}
\usepackage{array}
\usepackage{lmodern}
\usepackage[T1]{fontenc}

 \usepackage{tikz}
 \newcommand*\circled[1]{\tikz[baseline=(char.base)]{\node[shape=circle,draw,inner sep=1pt] (char) {#1};}}

\def\<{\langle}
\def\>{\rangle}
\def\Li{\text{\rm Li}}
\def\QQ{\mathbb{Q}}
\def\RR{\mathbb{R}}
\def\CC{\mathbb{C}}

\def\ZZ{\mathbb{Z}}
\def\PP{\mathbb{P}}
\def\VV{\mathbb{V}}
\def\cS{\mathcal{S}}
\def\cV{\mathcal{V}}

\def\vf{\varphi}

\def\fm{\mathfrak{m}}
\def\ad{\text{ad}}
\def\Ad{\text{Ad}}

\def\fg{\mathfrak{g}}

\def\k12{\mathcal{K}_{\lambda_1,\lambda_2}}
\def\tk12{\tilde{\mathcal{K}}_{\lambda_1,\lambda_2}}
\def\ck12{\check{\mathcal{K}}_{\lambda_1,\lambda_2}}

\def\cX{\mathcal{X}}
\def\cF{\mathcal{F}}
\def\hcF{\cF^{\bullet}}

\def\b{\bullet}

\theoremstyle{definition}

\theoremstyle{definition}

\makeatother

\usepackage{babel}
  \providecommand{\conjecturename}{Conjecture}
  \providecommand{\definitionname}{Definition}
  \providecommand{\examplename}{Example}
  \providecommand{\lemmaname}{Lemma}
  \providecommand{\propositionname}{Proposition}
  \providecommand{\remarkname}{Remark}
\providecommand{\theoremname}{Theorem}

\begin{document}

\title[Arithmetic of degenerating principal VHS]{Arithmetic of degenerating principal VHS: \\ examples arising from mirror symmetry and middle convolution}

\author[G. da Silva Jr., M. Kerr, and G. Pearlstein]{Genival da Silva Jr., Matt Kerr, and Gregory Pearlstein}

\subjclass[2000]{14D07, 14M17, 17B45, 20G99, 32M10, 32G20}
\begin{abstract}
We collect evidence in support of a conjecture of Griffiths, Green
and Kerr \cite{GGK} on the arithmetic of extension classes of limiting
mixed Hodge structures arising from semistable degenerations over
a number field. After briefly summarizing how a result of Iritani
\cite{Ir} implies this conjecture for a collection of hypergeometric
Calabi-Yau threefold examples studied by Doran and Morgan \cite{DM},
the authors investigate a sequence of (non-hypergeometric) examples
in dimensions $1\leq d\leq6$ arising from Katz's theory of the middle
convolution \cite{Ka,DR}. A crucial role is played by the Mumford-Tate
group (which is $G_{2}$) of the family of 6-folds, and the theory
of boundary components of Mumford-Tate domains \cite{KP}.
\end{abstract}
\maketitle

\section{Introduction}

Absolutely irreducible $\QQ$-local systems can underlie at most one
polarized variation of Hodge structure, which suggests that the asymptotics
of such variations at a puncture should exhibit interesting arithmetic.
For variations of motivic origin, one envisions arithmetic constraints
on the extension classes (periods) of the limiting mixed Hodge structures,
cf. Conjecture \ref{conj ggk} below. The Mumford-Tate group $G$
of the VHS imposes its own algebraic constraints upon these extensions,
which can simplify the form of the conjecture. Given a $G$-rigid
local system, the middle convolutions of Katz \cite{Ka} give some
hope for constructing a family of motives with the local system (and
VHS) appearing in its cohomology.

This paper was motivated by the desire to check the conjectural property
for some local systems on the thrice-punctured sphere underlying motivic
VHS of type $(1,1,1,1)$ and $(1,1,1,1,1,1,1)$, at a point of maximal
unipotent monodromy. Variations with extremal Hodge numbers $1$ have
been called ``Calabi-Yau'' for some time; when all the Hodge numbers
are $1$, terminology from representation theory (``principal $sl_{2}$'')
suggests calling them ``principal''. Robles's recent classification
\cite{Ro} of the corresponding Hodge representations rules out all
exceptional groups except for $G_{2}$ as Mumford-Tate group, which
itself can only occur in the weight/level 6 case. Moreover, the effects
of $G_{2}$ on the limiting MHS are well-understood via boundary components
\cite{KP}, a story which is briefly reviewed in $\S3$.

Our first main point is that a recent result of Iritani in mirror
symmetry \cite{Ir} allows one to compute the limiting extension classes
for many of the $(1,1,1,1)$ examples classified by Doran and Morgan
\cite{DM}. If $X^{\circ}$ is a complete intersection CY 3-fold in
a weighted projective space with rank 4 even cohomology, we use 
Iritani's $\hat{\Gamma}$ integral structure on its quantum cohomology to give
a straightforward computation of the (large complex structure) limiting
period matrix $\Omega_{lim}$ of the VHS arising from $H^{3}$ of
the mirror family (cf. \eqref{e:4*}). In particular, the nontorsion
extension class $\varepsilon\in\CC/\QQ(3)\cong Ext_{MHS}^{1}(\QQ(-3),\QQ(0))$
is given by $\left(\int_{X^{\circ}}c_{3}(X^{\circ})\right)\zeta(3)$.

The second point is that using middle convolution, one may construct
interesting motivic variations not treatable by mirror symmetry, but
where the limiting periods may be computed directly by a residue method.
This approach, which we apply in $\S5$ to two examples (including
one with $G_{2}$ monodromy) from the work of Dettweiler and Reiter
\cite{DR}, shows promise more generally for cyclic covers branched
along a union of hyperplanes. Moreover, it gives a clearer picture
of the origin of the zeta values in limiting extension classes, which
in $\S$4 is buried in a deep mirror theorem. The main idea is that
Katz's method builds a sequence of families $X_{d}(t)$ ($d\geq1$)
of the form $w^{2}=f(x_{1},\ldots,x_{d},t)$, whose $\log^{d}t$ period
in a neighborhood of the point $t=0$ of maximal unipotency is given
by the iterative formula
\[
\pi_{2j}(t):=i\int_{t}^{1}\frac{\pi_{2j-1}(x)\, dx}{\sqrt{x(1-x)(1-\frac{t}{x})}},\;\;\;\;\pi_{2j+1}(t):=i\int_{t}^{1}\frac{\pi_{2j}(x)\, dx}{x\sqrt{1-\frac{t}{x}}},
\]
where $\pi_{1}(t):=2\int_{1}^{\frac{1}{t}}\frac{dx}{\sqrt{(1-tx)(1-x)x}}$.
The top weight graded piece of $H^{d}(X_{d}(t))$ contains a principal
variation (at least, for $d\leq7$), and all the data of $\Omega_{lim}$
for this VHS is contained in the asymptotics of $\pi_{d}(t)$.

For $d=3$, we are able to completely determine these asymptotics
(Theorem \ref{thm:d=00003D3}), and hence the extension class $\varepsilon=-48\zeta(3)\in\CC/\QQ(3)$.
Our luck did not hold out for $d=6$ ($\S5.5$), where we were only
able to compute ``part'' of the integral; however, this piece does
contain a term of the form $-72\zeta(5)$ as expected (``towards''
the extension class in $Ext_{MHS}^{1}(\QQ(-5),\QQ(0))$), and Conjecture
\ref{conj:d=00003D6} represents an educated guess at the entire thing.
Moreover, our partial computation contains a ravishing number-theoretic
tidbit \inputencoding{latin1}{(Lemma \ref{lem:g2=00003D})}\inputencoding{latin9},
which we call the \emph{$G_{2}$-identity}\inputencoding{latin1}{:}

\inputencoding{latin9}\tiny
\[
\sum{}^{'}\frac{(\frac{1}{2})_{k_{1}}(\frac{1}{2})_{k_{2}}(\frac{1}{2})_{a}(\frac{1}{2})_{b}}{(b-a+\frac{1}{2})(b+k_{1}+\frac{1}{2})(a+k_{1})(a+k_{2})}+\sum{}^{'}\frac{(\frac{1}{2})_{k_{1}}(\frac{1}{2})_{k_{2}}(\frac{1}{2})_{a}(\frac{1}{2})_{b}}{(b-a+\frac{1}{2})(b+k_{1}+\frac{1}{2})(b+k_{2}+\frac{1}{2})(a+k_{1})}
\]
\[
-\sum{}^{'}\frac{(\frac{1}{2})_{a}(\frac{1}{2})_{b}}{(b+a+\frac{1}{2})(b+\frac{1}{2})a^{2}}\;\;\;\;\;\;\;\;=\;\;\;\;\;\;\;\;\frac{64\pi}{3}\log^{3}2+\frac{2\pi^{3}}{3}\log2-12\pi\zeta(3).
\]
\normalsize (Here $\sum{}^{'}$ means to sum over tuples of non-negative
integers for which the summand is defined.) This identity is \emph{a
consequence of the vanishing of the }\inputencoding{latin1}{\emph{``third
extension class'' in the maximal unipotent LMHS of a principal variation
}}\inputencoding{latin9}\emph{with $G_{2}$ Mumford-Tate group}, and
(bizarrely) it is needed to finish off the $d=3$ computation. Our
computations also make heavy use of some identities relating hypergeometric
special values and Riemann zeta values; how to derive these is outlined
in the Appendix.


In writing this paper we encountered several questions which merit
further investigation. For example, is there a direct construction
of ``limiting data'', from a quasi-unipotent $G$-rigid local system
$\mathbb{V}$ over $\PP^{1}\backslash\text{pts.}$, that does not
pass through variations of Hodge structure? Here the motivation is
that there should exist a unique VHS underlain by $\mathbb{V}$, cf.
Remark \ref{simpson conj}. In the middle convolution case, is there
a better approach to computing the LMHS than direct computation of
asymptotics of $\pi_{d}(t)$, perhaps one extending the computation
of the LMHS Hodge numbers in \cite{DS}? Finally, can one use mirror
symmetry to compute any limiting extension classes in $Ext_{MHS}^{1}(\QQ(-5),\QQ(0))$?
Here the problem is typically that some extension class in $Ext_{MHS}^{1}(\QQ(-3),\QQ(0))$
is nonzero, and then one of the previous form is not well-defined;
but even in this case it would still be of interest to compute $\Omega_{lim}$
(say, for CY 5- or 6-folds).

\subsubsection*{Acknowledgments:}

MK thanks C. Doran, P. Griffiths and Y. Kovchegov for interesting discussions
regarding this paper. We also wish to acknowledge partial support
from CAPES (da Silva Jr.), NSF Grants DMS-1068974 and DMS-1361147
(Kerr), and NSF Grants DMS-1002625 and DMS-1361120 (Pearlstein).

\section{Local systems and limiting mixed Hodge structures}

Let $\cS$ be a complex algebraic manifold, and fix a base point $s_{0}\in\cS$.
Given an absolutely irreducible $\QQ$-local system $\VV$ over $\cS$,
put $V:=\VV_{s_{0}}$ and define the monodromy group
\[
\Gamma:=\text{image}\left(\rho:\pi_{1}(\cS,s_{0})\to GL(V)\right).
\]
The geometric monodromy group $\Pi$ is the identity connected component
of its $\QQ$-Zariski closure.

Now suppose there exists a polarized variation of Hodge structure
(PVHS) $\cV=\left(\VV,Q,\hcF\right)$ over $\VV$ of weight $n$.%
\footnote{$\mathcal{V}$ will also sometimes denote, by abuse of notation, the
locally free sheaf $\mathbb{V}\otimes\mathcal{O}_{\mathcal{S}}$ (or
the corresponding vector bundle).%
}
\begin{prop}
\label{prop uniqueness}Up to Tate twist, $\cV$ is unique.\end{prop}
\begin{proof}
Given $\cV$, $\cV'$ PVHS over $\VV$, $\cV^{\vee}\otimes\cV'$ is
a PVHS over $\VV^{\vee}\otimes\VV$, and by Schur's lemma $\left(\VV^{\vee}\otimes\VV\right)^{\Gamma}=\QQ\langle\text{id}_{\VV}\rangle$.
By the Theorem of the Fixed Part \cite{Sc}, $\QQ\langle\text{id}_{\VV}\rangle$
therefore underlies a constant sub-VHS of $\cV^{\vee}\otimes\cV$'
, rank $1$ hence of type $(p,p)$.
\end{proof}
As Griffiths puts it, Riemann would be proud: this sort of result
goes back to his characterization of the hypergeometric functions
by their local monodromy about $0,1,\infty$. Note that the existence
of $\mathcal{V}$ implies that $\mathbb{V}$ is semisimple with quasi-unipotent
monodromies.

Assume now that $\cS$ has a smooth compactification $\bar{\cS}$
with a holomorphic disk embedding
\[
\begin{array}{ccccc}
\cS & \subset & \bar{\cS} & \ni & x\\
\cup &  & \cup &  & \uparrow\\
\Delta^{*} & \overset{\jmath}{\subset} & \Delta & \ni & 0
\end{array}
\]
and let $s$ be a choice of local coordinate on $\Delta$. Restricting
$\cV$ to $\Delta^{*}$, we assume the local monodromy $T$ is unipotent
and set
\[
N:=\log(T)=\sum_{k\geq1}\frac{(-1)^{k-1}}{k}(T-I)^{k}\in End(V,Q).
\]
There exists a unique increasing filtration $W_{\bullet}=W(N)_{\bullet}$
on $V$ such that ($\forall k$)
\[
\left\{ \begin{array}{c}
N(W_{k}V)\subset W_{k}V\\
N^{\ell}:Gr_{n+k}^{W}V\overset{\cong}{\to}Gr_{n-k}^{W}V
\end{array}\right..
\]
Moreover, the ``untwisted'' local system $\tilde{\VV}:=\jmath_{*}(e^{-\ell(s)N}\VV)$
(where $\ell(s):=\frac{\log(s)}{2\pi i}$) extends to $\Delta$. By
the Nilpotent Orbit Theorem \cite{Sc}, the Hodge sheaves $\cF^{p}\subset\cV$
extend to locally free subsheaves $\cF_{e}^{p}\subset\cV_{e}:=\tilde{\VV}\otimes\mathcal{O}_{\Delta}$
on $\Delta$, and the $SL_{2}$-orbit Theorem {[}op. cit.{]} implies:
\begin{prop}
$(\psi_{s}\cV)_{x}:=\left(\tilde{\VV},W_{\b},\cF_{e}^{\b}\right)|_{x}$
is a mixed Hodge structure polarized by N, called the limiting mixed
Hodge structure (LMHS).
\end{prop}
Writing $F_{lim}^{\b}:=\cF_{e}^{\b}|_{x}$, we shall denote by $\cV_{nilp}:=\left(\VV,Q,e^{-\ell(s)N}F_{lim}^{\b}\right)$
the associated nilpotent orbit, which is again a VHS over (a possibly smaller)
punctured disk $\Delta^*$.

We conclude that to an absolutely irreducible local system on $\cS$,
point $x\in\bar{\cS}$, and local coordinate $s\in\mathcal{O}(\Delta)$,
Schmid's results associate a MHS. The extension classes inherent in
the latter are thereby already in this sense \emph{invariants of the
local system}.
\begin{rem}
\label{simpson conj}Of course, this begs the question as to which
local systems underlie a PVHS! It is expected (cf. \cite{DR2}) that
quasi-unipotency and $G$-rigidity%
\footnote{$\mathbb{V}$ is \emph{$G$-rigid} if the $G$-orbit of the associated
monodromy representation $\rho$ is open in $Hom(\pi_{1}(\cS),G)$.%
} (for some $G\leq GL(V)$ containing $\Gamma$) suffice for $\mathbb{V}$
to underlie a \emph{motivic} PVHS, that is, one arising from a family
of varieties over $\cS$. For $\cS=\PP^{1}\backslash\text{pts.}$
and $G=GL(V)$, this is proved by Katz \cite{Ka} using his middle
convolution algorithm, which we touch on in $\S5$.
\end{rem}
Recall next that a motive over a field $k\subseteq\mathbb C$ is, roughly 
speaking, a bounded complex of smooth quasi-projective varieties with arbitrary
maps between them, all defined over $k$. Through a ``realization''
process similar to hypercohomology, one can take the various cohomology
groups of such a complex, which yields in particular (from de Rham
and Betti) a MHS we will say to be \emph{$k$-motivated}.

Now assume that $\VV=R^{k}f_{*}\QQ_{\cX}$ arises from the following
situation, called a \emph{semi-stable degeneration (SSD) over $k$}:
\[
\begin{array}{cccccc}
\cX & \hookrightarrow & \bar{\cX} & \hookleftarrow & X_{0} & =\bar{f}^{-1}(x)\\
\downarrow f &  & \downarrow\bar{f} &  & \downarrow\\
\cS & \hookrightarrow & \bar{\cS} & \ni & x\\
 &  & \downarrow g\\
 &  & \PP^{1}
\end{array}
\]
where $\bar{f}$ is proper and flat, $f$ is smooth, $g(x)=0$ with
$\text{ord}_{x}(g)=1$, and $X_{0}:=(g\circ\bar{f})^{-1}(0)=\cup Y_{i}$
is a reduced SNCD. (That is, the intersections $Y_{I}:=\cap_{i\in I}Y_{i}$
are smooth and transversal.) Moreover, the entire diagram, and all
inclusions $Y_{I}\hookrightarrow Y_{I\backslash i}$, are defined
over $k$.
\begin{conjecture}
\cite{GGK} \label{conj ggk}Let $s$ be the restriction of $g$ to
a disk about $x$. Then $(\psi_{s}\cV)_{x}$ is $k$-motivated. In
particular, the extension class $\varepsilon\in\CC/\ZZ(m)$ of any
Tate subquotient $0\to\ZZ(0)\to\mathbb{E}\to\ZZ(-m)\to0$ belongs
to the image of motivic cohomology $H_{M}^{1}(Spec(k),\mathbb{Z}(m))$
under the generalized Abel-Jacobi mapping.
\end{conjecture}
If $k=\QQ$, this says that 
\[
\varepsilon\equiv\left\{ \begin{array}{cc}
\log(a),\, a\in\QQ^{*} & (m=1)\\
q\zeta(m),\, q\in\QQ & (m>1)
\end{array}\right.;
\]
note that $\zeta(m)\in\CC/\ZZ(m)$ is torsion if $m$ is even. In
what follows we shall be working with rational coefficients, and hence
interested only in odd $m$.
\begin{rem}
Several constructions of limiting motives have recently appeared in
the literature, for instance \cite{Le}. Conjecture \ref{conj ggk}
would probably follow from the assertion (itself still conjectural)
that the Hodge realization of Levine's motive is the LMHS.
\end{rem}
The existence (up to torsion) of a Tate subquotient is an algebraic
requirement; it is useful at this point to consider what algebraic
conditions might produce it, which brings us to the next section.

\section{Mumford-Tate domains and boundary components}

Let $V$ be a (finite-dimensional) $\QQ$-vector space, $Q:V\times V\to\QQ$
a $(-1)^{n}$-symmetric nondegenerate bilinear form. Writing $S^{1}<\CC^{*}$
for the unit circle, consider a homomorphism
\[
\vf:S^{1}\to SL(V_{\RR})
\]
with $\vf(-1)=(-1)^{n}\cdot\text{id}_{V}$.
\begin{defn}
\label{Def PHS}(i) $(V,Q,\vf)$ is a \emph{polarized Hodge structure}
\emph{(PHS)}%
\footnote{It is implicit here that the PHS is ``pure of weight $n$'', even
though the definition only records the parity of $n$.%
} if
\[
\left\{ \begin{array}{ccc}
\text{(I)} & \vf(S^{1})\subset Aut(V,Q) & \text{and}\\
\text{(II)} & Q(v,\vf(i)\bar{v})>0 & (\forall v\in V_{\CC}\backslash\{0\}).
\end{array}\right.
\]

(ii) The \emph{Mumford-Tate group} \emph{(MTG)} $G_{\vf}$ is the
$\QQ$-algebraic group closure of $\vf(S^{1})$.
\end{defn}
The $\vf(S^{1})$-fixed points in the tensor spaces $T^{k,l}V:=V^{\otimes k}\otimes(V^{\vee})^{\otimes l}$
are the \emph{Hodge tensors} ${Hg}^{k,l}V$.
\begin{prop}
\cite{De} $G_{\vf}$ is the subgroup of $GL(V)$ fixing $\oplus Hg^{k,l}V$
pointwise.
\end{prop}
The Lie group of real points of the MTG acts by conjugation on $\vf$,
and the (connected) Mumford-Tate domain (MTD) associated to $\vf$
is the orbit (under the identity connected component)
\[
D:=G_{\vf}(\RR)^{+}.\vf\cong G_{\vf}(\RR)^{+}/\mathcal{H}_{\vf}.
\]
By taking $V_{'\vf}^{p,n-p}\subset V_{\CC}$ to be the $z^{p-q}$-eigenspace
for $'\vf(z)$ ($'\vf\in D$), $D$ may be viewed as (a connected
component of) the locus in some period domain on which a finite set
of tensors in $\oplus T^{k,l}V$ becomes Hodge. Moreover, using the
Hodge flags $F_{'\vf}^{\b}V_{\CC}=\oplus_{p\geq\b}V_{'\vf}^{p,n-p}$,
we may embed $D$ in a product of Grassmanians, where its Zariski
closure defines the compact dual
\[
\check{D}=G_{\vf}(\CC).F_{\vf}^{\b}\cong G_{\vf}(\CC)/P_{F_{\vf}^{\b}}.
\]

For a polarized variation $\cV$ as in $\S2$, the pointwise MTG is
equal to some $G<GL(V)$ on the complement of a countable union of
proper analytic subvarieties, and we call this the MTG of $\cV$.
Since $\Pi\trianglelefteq G^{der}$ \cite{An}, we obtain (after possibly
replacing $\cS$ by a finite cover) a period map 
\[
\Phi:\cS\to\Gamma\backslash D.
\]
For studying the possible LMHS, it is convenient to replace $G$ by
$G^{ad}$, $V$ by $\fg=Lie(G^{ad})$, $\vf$ by the composition $S^{1}\overset{\vf}{\to}G\overset{Ad}{\twoheadrightarrow}G^{ad}$,
$Q$ by $-B$ ($B=$ Killing form), and the weight $n$ by $0$; then
$D$ is unchanged and no information is lost \cite{KP}. Assume now
that this change has been made, and let $N\in\fg_{\QQ}\backslash\{0\}$
be a nilpotent element (acting on $V$ by $\ad$).
\begin{defn}
\cite{KP} The \emph{pre-boundary component}%
\footnote{It sometimes happens that $\tilde{B}(N)$ as defined is not connected;
in this case, one should replace it by a choice of connected component.%
} 
\[
\tilde{B}(N):=\left\{ F^{\b}\in\check{D}\left|\begin{array}{cc}
NF^{\b}\subset F^{\b-1},\\
\Ad(e^{\tau N})F^{\b}\in D & \text{for }\Im(\tau)\gg0
\end{array}\right.\right\} 
\]
associated to $N$ and $D$ classifies the possible LMHS $(F^{\b},W(N)_{\b})$
of period maps into any quotient $\Gamma\backslash D$. Let $\tilde{B}_{\RR}(N)\subset\tilde{B}(N)$
be the subset consisting of $\RR$-split LMHS, and $B(N):=e^{\CC N}\backslash\tilde{B}(N)$
the set of nilpotent orbits. The \emph{boundary component} associated
to $N$, $D$, and a choice of $\Gamma$ is then $\bar{B}(N):=\Gamma_{N}\backslash B(N)$,
where $\Gamma_{N}:=\text{Stab}(\CC N)\cap\Gamma$.
\end{defn}

\begin{rem}
One can think of the quotient by $e^{\CC N}$ as eliminating the dependence of the LMHS on the scaling of the local coordinate.  We will be interested below in computing the point in $\bar{B}(N)$ associated to LMHS of Hodge-Tate type for principal VHS $\mathcal{V}$, so that $Gr^{W(N)}_k V$ is $0$ for $k$ odd and $\QQ(-\frac{k}{2})$ or $0$ for $k$ even.  In this case, there is an easy first step: we may rescale (``canonically normalize'') the local coordinate to eliminate the adjacent extensions (of $\QQ(p)$ by $\QQ(p+1)$) in $V$.
\end{rem}
\noindent The terminology of ``boundary component'' comes from the appearance of $\bar{B}(N)$ in partial
compactifications of $\Gamma\backslash D$, assuming it is nonempty.
In this case, let $M\leq Aut(G,B)$ be the subgroup fixing all Hodge
tensors of all LMHS in $\tilde{B}(N)$. (For a given MHS on $V$,
the Hodge tensors are the elements of $\oplus_{p,k,l}Hom_{MHS}(\QQ(p),T^{k,l}V)$.)
\begin{prop}
\cite{DK} Let $Z_{G}(N)\leq G$ denote the centralizer of $N$. Then
$M\leq Z_{G}(N)$ and $Z_{G}(N)$ share the unipotent radical $U=\exp\{\text{im}(\ad N)\cap\ker(\ad N)\}$.
Writing $G_{N}$ for a choice of Levi subgroup of $M(=G_{N}U$), $M(\RR)$
resp. $G_{N}(\RR)\ltimes U(\CC)$ act transitively on $\tilde{B}_{\RR}(N)$
resp. $\tilde{B}(N)$. Assuming $\Gamma$ is neat, there exists an
iterated generalized intermediate Jacobian fibration
\[
\bar{B}(N)\twoheadrightarrow\cdots\twoheadrightarrow\bar{B}(N)_{(k)}\twoheadrightarrow\cdots\twoheadrightarrow\bar{B}(N)_{(1)}\twoheadrightarrow\bar{D}(N)=\Gamma_{N}\backslash D(N),
\]
where $D(N)$ is a MTD for $G_{N}$.
\end{prop}
We will need a source of PHS $\vf$ with interesting MTG. Let $G$
be a $\QQ$-simple adjoint group of rank $r$ such that $G_{\RR}$
has a compact Cartan subgroup. Let $\theta\in Aut(G_{\RR})$ be a
Cartan involution, $K<G_{\RR}$ the corresponding maximal compact
subgroup, and $T_{\RR}<K$ a Cartan subgroup (of dimension $r$).
Writing $\mathcal{R}$ for the root lattice, $\Delta$ {[}resp. $\Delta_{c}$,
$\Delta_{n}${]} for the roots {[}resp. compact, noncompact roots{]},
we have the Cartan decomposition 
\[
\fg=\mathfrak{t}\oplus\left(\oplus_{\alpha\in\Delta_{c}}\fg_{\alpha}\right)\oplus\left(\oplus_{\beta\in\Delta_{n}}\fg_{\beta}\right).
\]

\begin{prop}
\cite{GGK2} Given a homomorphism $\pi:\mathcal{R}\to2\ZZ$ with $\pi(\Delta_{c})\in4\ZZ$,
$\pi(\Delta_{n})\subset4\ZZ+2$, there exists a unique weight $0$
PHS $(\fg,-B,\vf)$ with $\frac{d\vf}{dz}(1)|_{\mathcal{R}}=\pi$.
Moreover, provided $T_{\RR}$ (hence also $\theta)$ is sufficiently
general, the MTG $G_{\vf}=G$. The Hodge numbers of the Hodge structures
on $\fg$ parametrized by $D=G(\RR)^{+}.\vf$ are then 
\[
h^{j,-j}=\left\{ \begin{array}{cc}
|\pi^{-1}(2j)|, & j\neq0\\
|\pi^{-1}(0)|+r, & j=0
\end{array}\right..
\]

\end{prop}
In some cases, these constructions ``lift to a standard representation''.
Here are two examples where this occurs, together with choices of
$N$ (with ``maximally unipotent'' $T=e^{N}$) which produce nonempty
boundary components, and the pictures of $(p,q)$ types for the resulting
LMHS.
\begin{example}
($\fg=\mathfrak{sp}_{4}$) With $\pi$ as shown below, $D$ (of dimension
$4$) parametrizes weight $3$ PHS (of rank $4$) on the standard
representation $V$ with Hodge numbers $(1,1,1,1)$: \[\includegraphics[scale=0.5]{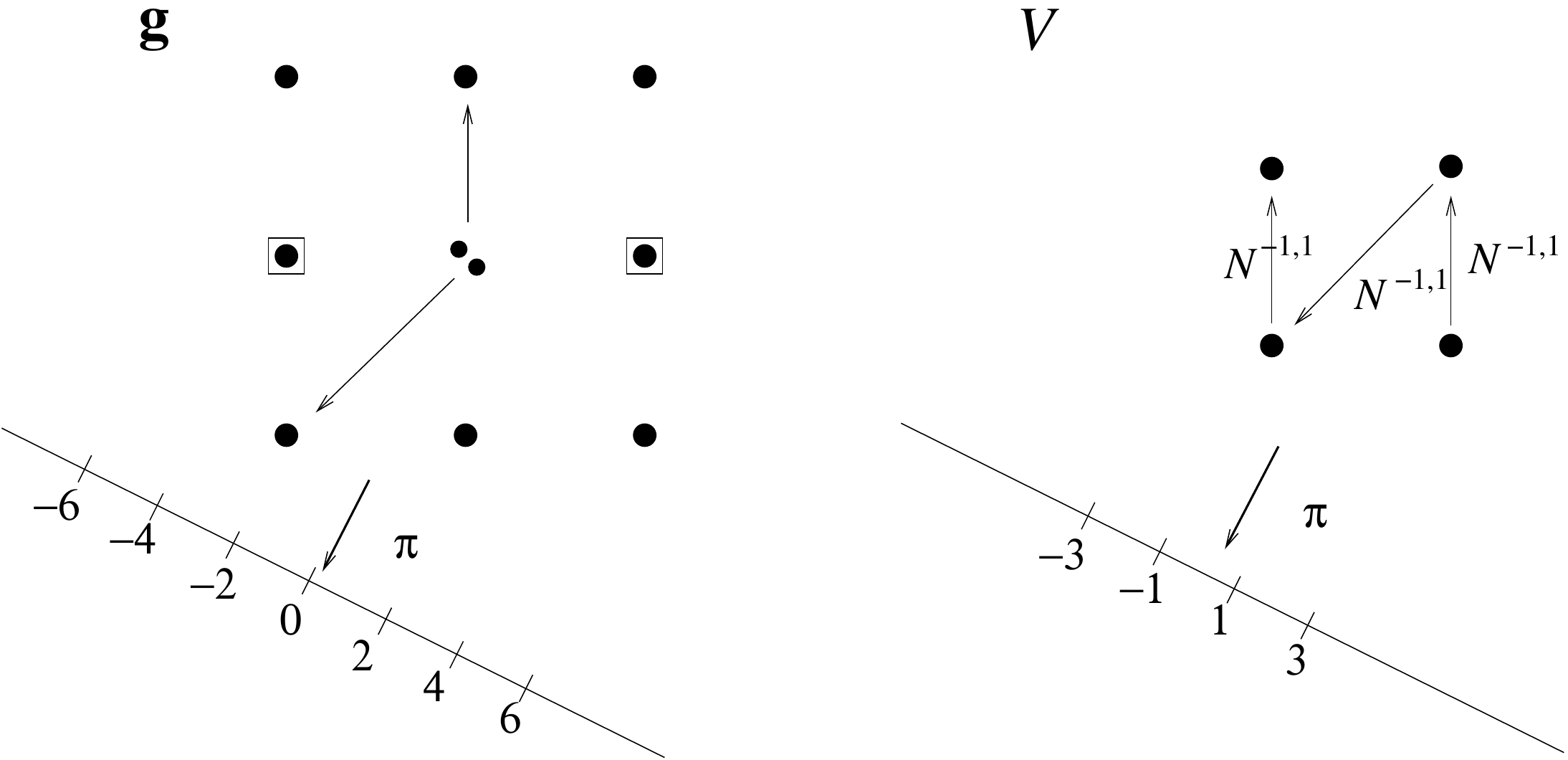}\](The
large dots on the left-hand side are roots, with the boxed ones in
$\Delta_{c}$; the small dots indicate the Cartan subalgebra $\mathfrak{t}$.
On the right, the dots are weights of $V$.) Moreover, there exist
$\vf\in D$ and $N\in\fg_{\QQ}$ such that $F_{\vf}^{\b}$ belongs
to $\tilde{B}(N)$, and the $\fg_{\vf}^{-1,1}$-component $N^{-1,1}$
is a linear combination of root vectors as shown on the left-hand
side (the arrows point to the corresponding roots) and operates on
$V$ as described by the arrows on the right-hand side. The LMHS $(F_{\vf}^{\b},W(N)_{\b})$
on $\fg$ and the induced LMHS on $V$ take the following form, where
the arrows describe the action of $N$:\[\includegraphics[scale=0.5]{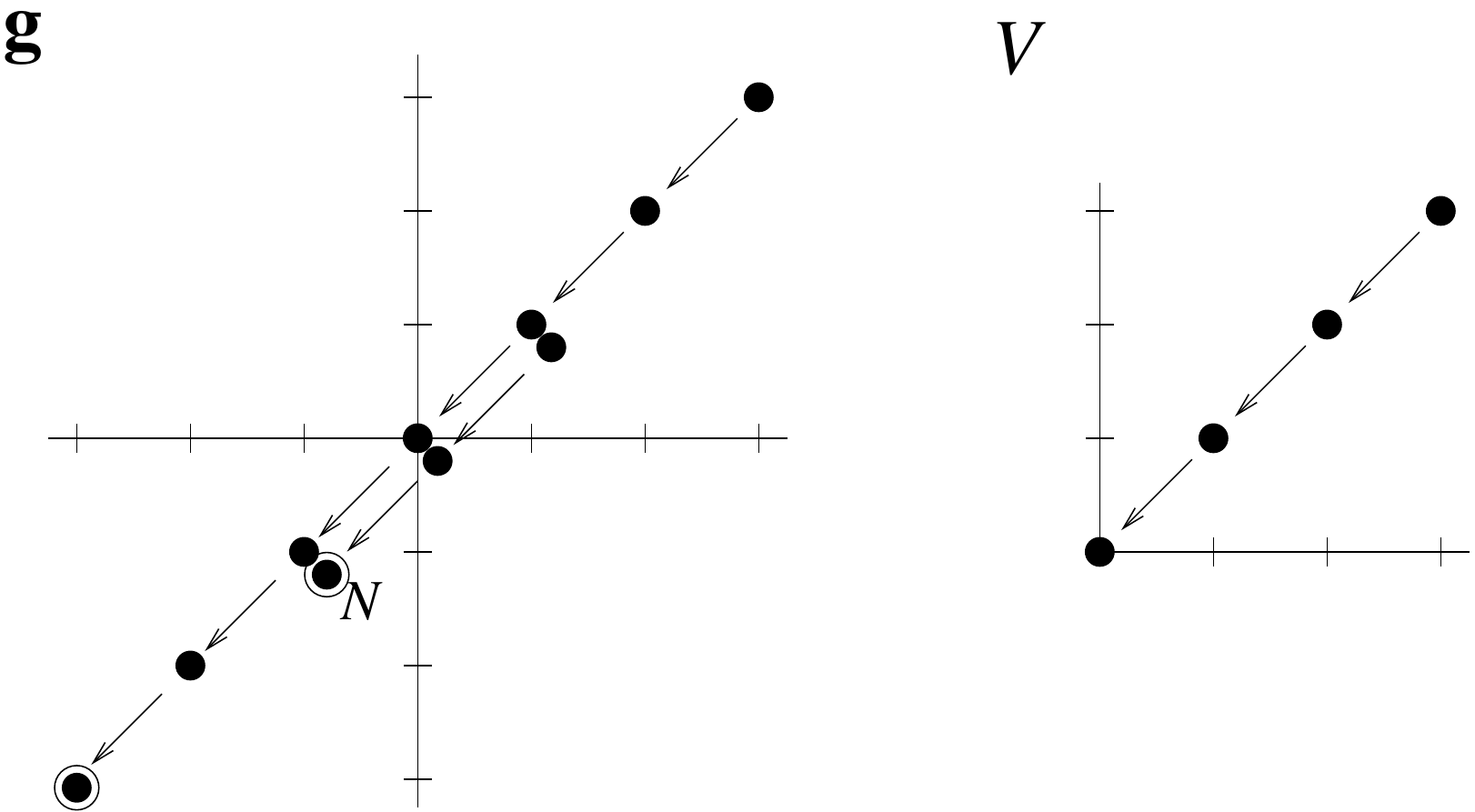}\]From
the picture one sees that $\fm_{\QQ}/\QQ N$ ($\fm=Lie(M)$ corresponding
to the circled types) is pure of rank one and type $(-3,-3)$; according
to \cite{KP} it follows that $\bar{B}(N)\cong Ext_{MHS}^{1}(\ZZ(-3),\ZZ(0))\cong\CC/\ZZ(3)$.
(Dividing out by $(2\pi i)^{3}$, this is just $\CC/\ZZ$.)
\end{example}
Turning to an exceptional group, we have
\begin{example}
($\fg=\fg_{2}$) Here $D$ is of dimension $5$, parametrizing weight
$6$ PHS on the $7$-dimensional irreducible representation $V$ of $G_{2}$ with Hodge
numbers $(1,1,1,1,1,1,1)$.\[\includegraphics[scale=0.5]{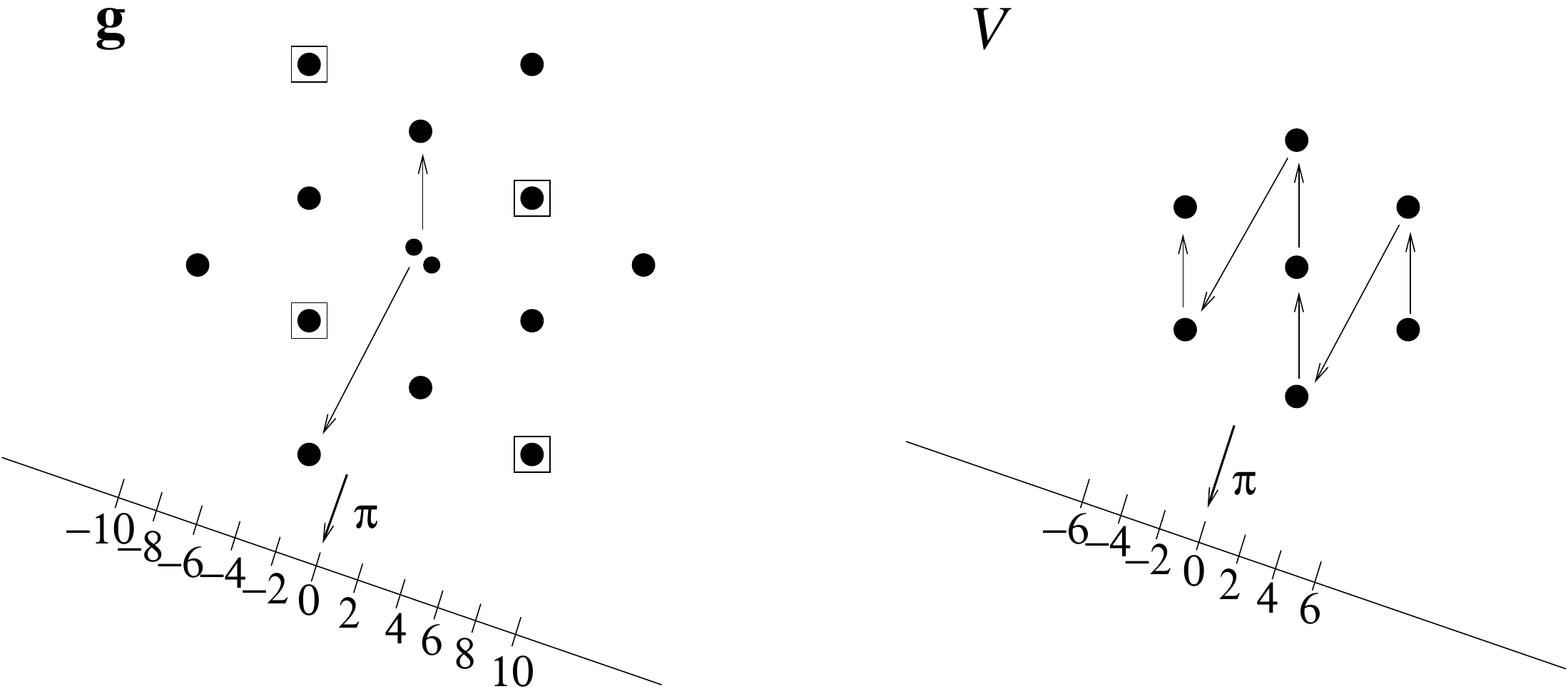}\]The
figures have the same meaning as in the $\mathfrak{sp}_{4}$ example.
For the LMHS one has:\[\includegraphics[scale=0.4]{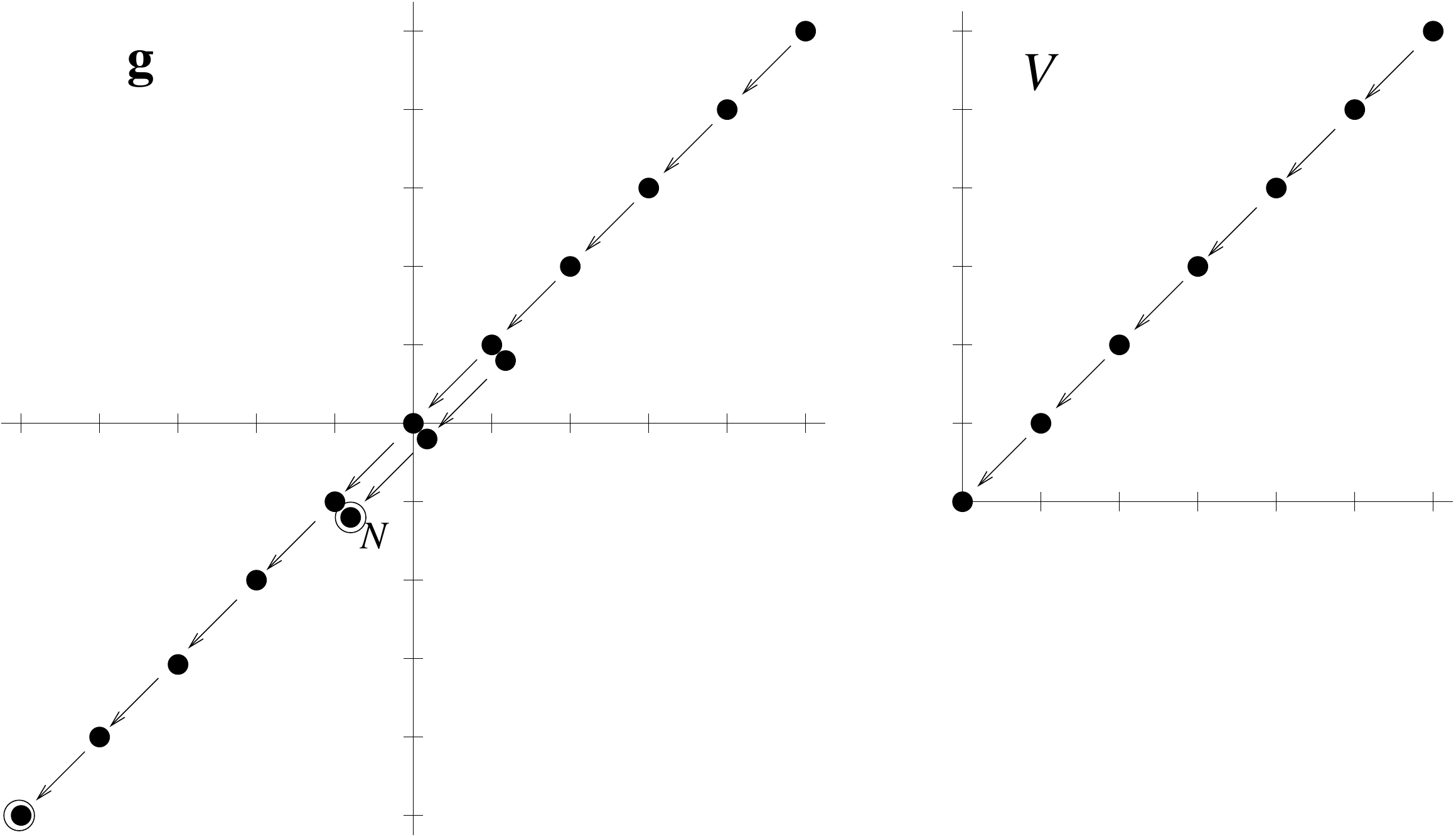}\]from
which $\bar{B}(N)\cong Ext_{MHS}^{1}(\ZZ(-5),\ZZ(0))\cong\CC/\ZZ(5)$.
\end{example}
Henceforth we shall be interested in the Hodge structures (and LMHS)
on $V$ rather than $\fg$. Note that in both examples, these PHS
are ``Calabi-Yau'' in the sense that the leading Hodge number is
$1$. Moreover, $\fg^{-1,1}$ has rank $2$ and is nonabelian; therefore
Griffiths transversality forces the image of a period map into $\Gamma\backslash D$
to be a curve. 

Most importantly, they each give fertile testing-grounds for Conjecture
\ref{conj ggk}. In $\S\S$4-5, we shall verify it (at $x=0$) for
some VHS over $\PP^{1}\backslash\{0,1,\infty\}$ arising from SSD's
over $\QQ$, which have MTG $Sp_{4}$ or $G_{2}$ and maximal unipotent
monodromy about $0$. In both cases this boils down to checking that
a single limiting period $\xi\in\CC/\QQ$ takes a particular form.

In the $Sp_{4}$ case, we may assume given a symplectic basis, so
that the polarization takes the form \begin{equation}\label{e:Q}Q=\left(\begin{array}{cccc}0 & 0 & 0 & 1\\0 & 0 & 1 & 0\\0 & -1 & 0 & 0\\-1 & 0 & 0 & 0\end{array}\right) .\end{equation}
After conjugating by $Sp_{4}(\QQ)$ to have \begin{equation}\label{e:N}N=\left(\begin{array}{cccc}0 & 0 & 0 & 0\\a & 0 & 0 & 0\\e & b & 0 & 0\\f & e & -a & 0\end{array}\right) \end{equation}
and canonically normalizing the local coordinate at $0$, one knows
(cf. \cite{GGK}) that the limiting period matrix takes the form
\[
\Omega_{lim}=\left(\begin{array}{cccc}
1 & 0 & 0 & 0\\
0 & 1 & 0 & 0\\
\frac{f}{2a} & \frac{e}{a} & 1 & 0\\
\xi & \frac{f}{2a} & 0 & 1
\end{array}\right)\;\;\;\;(\xi\in\CC).
\]
The entries other than $\xi$ are rational and correspond to torsion
extension classes. The LMHS is $\QQ$-motivated if and only if $\xi=q\frac{\zeta(3)}{(2\pi i)^{3}}$
($q\in\QQ$).

For $G_{2}$, again after appropriate normalizations, one has
\[
\Omega_{lim}=\left(\begin{array}{ccccccc}
1 & 0 & 0 & 0 & 0 & 0 & 0\\
0 & 1 & 0 & 0 & 0 & 0 & 0\\*
* & * & 1 & 0 & 0 & 0 & 0\\*
\circled{*} & * & * & 1 & 0 & 0 & 0\\*
* & \circled{*} & * & * & 1 & 0 & 0\\
\xi & * & \circled{*} & * & * & 1 & 0\\*
* & \xi & * & \circled{*} & * & 0 & 1
\end{array}\right)
\]
where $*$ denotes rational numbers. In the $\QQ$-motivated scenario,
$\xi=q\frac{\zeta(5)}{(2\pi i)^{5}}$ ($q\in\QQ$). For the same type
of LMHS on $V$ but with the larger M-T group $SO(3,4)$ (instead
of $G_{2}$), the third extension class need not be trivial.That is,
it is $G_{2}$ which forces the circled entries to be rational.

\section{Calabi-Yau variations from mirror symmetry}

In this section we shall briefly describe how a recent result of Iritani
\cite{Ir} allows one to systematically compute LMHS of variations
arising from families of anticanonical toric complete intersections.
We shall carry this out for the 1-parameter, $h^{2,1}=1$ hypergeometric
families of complete intersection C-Y threefolds classified in \cite{DM}.
Each family yields a semistable degeneration over $\QQ$ (cf. $\S2$)
with $X_{0}$ the (suitably blown-up) ``large complex structure limit''
fiber.

Until recently, toric mirror symmetry (e.g., as described in \cite{CK}
or \cite{Mo}) only identified \emph{complex} variations of Hodge
structure arising from the A-model and B-model, because the Dubrovin
connection on quantum cohomology merely provides a $\CC$-local system
on the A-model side. Iritani's mirror theorem says that the integral
structure on this local system provided by the $\hat{\Gamma}$-class
(in the sense described below) completes the A-model $\CC$-VHS to
a $\ZZ$-VHS matching the one arising from $H^{3}$ of fibers on the
B-model side. The upshot is that to compute $\Omega_{lim}$ (at $0$)
for a $1$-parameter family of toric complete intersection Calabi-Yau
3-folds $X_{t}\subset\PP_{\Delta}$
over $\PP^{1}\backslash\{0,1,\infty\}$, we may use what boils down
to characteristic class data from the mirror $X^{\circ}\subset\PP_{\Delta^{\circ}}$.

In each case, $V:=H^{even}(X^{\circ},\CC)=\oplus_{j=0}^{3}H^{j,j}(X^{\circ})$
is a vector space of rank $4$, $\PP:=\PP_{\Delta^{\circ}}=\mathbb{WP}(\delta_{0},\ldots,\delta_{3+r})$
is a weighted projective space%
\footnote{Technically, there are three exceptions to this amongst the examples
we consider, which are weighted projective spaces $\mathbb{WP}(\delta_{0},\ldots,\delta_{n})$
for which the convex hull of $\{e_{1},\ldots,e_{n},-\sum\delta_{i}e_{i}\}$
is not reflexive. As described in \cite{DM}, taking $\Delta$ to
be the convex hull of this set together with $-e_{n}$ yields a reflexive
polytope, and $\PP_{\Delta^{\circ}}$ is the blow-up of the $\mathbb{WP}$
at a point not meeting (hence not affecting) the complete intersections
we consider. Hence we may take $X^{\circ}\subset\mathbb{P}=\mathbb{WP}(\delta_{1},\ldots,\delta_{n})$.%
} (with $\delta_{0}=\delta_{1}=1$), and $X^{\circ}\subset\mathbb{P}$
is smooth%
\footnote{The codimension of the singular locus in $\PP$ is at least $4$ in
every case, so does not meet a sufficiently general $X^{\circ}$.%
} of multidegree $(d_{k})_{k=1}^{r}$ with $\sum d_{k}=\sum\delta_{i}=:m$.
Let $H$ denote the intersection with $X^{\circ}$ of the vanishing
locus of the weight 1 homogeneous coordinate $X_{0}$; write $\tau[H]\in H^{1,1}(X^{\circ})$
for the K\"ahler class and $q=e^{2\pi i\tau}$ for the K\"ahler
parameter. We shall give a general recipe (following \cite[sec. 1]{DK})
for constructing a polarized $\ZZ$-VHS, over $\Delta^{*}:\;0<|q|<\epsilon$,
on $\mathcal{V}:=V\otimes\mathcal{O}_{\Delta^{*}}$.

The easy parts are the Hodge filtration and polarization. Indeed,
we simply put $F^{p}:=\oplus_{j\leq3-p}H^{j,j}\subset V$ and $\mathcal{F}_{e}^{p}:=F^{p}\otimes\mathcal{O}_{\Delta}\subset V\otimes\mathcal{O}_{\Delta}=:\mathcal{V}_{e}.$
Similarly, $Q$ on $\mathcal{V}_{e}$ is induced from the form on
$V$ given by the direct sum of pairings $Q_{j}:H^{j,j}\times H^{3-j,3-j}\to\CC$
defined by $Q_{j}(\alpha,\beta):=(-1)^{j}\int_{X^{\circ}}\alpha\cup\beta$.
A Hodge basis $e=\{e_{i}\}_{i=0}^{3}$ of $H^{even}$, with $e_{i}\in H^{3-i,3-i}(X^{\circ})$
and $[Q]_{e}$ of the form \eqref{e:Q}, is given by $e_{3}=[X^{\circ}]$,
$e_{2}=[H]$, $e_{1}=-[L]$, and $e_{0}=[p]$. Here $L$ is a copy
of $\PP^{1}$ (parametrized by $[X_{0}:X_{1}]$) in $X^{\circ}$ with
$L\cdot H=p$, and $[H]\cdot[H]=m[L]$. The $\{e_{i}\}$ give a Hodge
basis%
\footnote{Note: in all bases we shall run the indices backwards ($e=\{e_{3},e_{2},e_{1},e_{0}\}$,
etc.) for purposes of writing matrices.%
} for $\mathcal{V}_{e}$.

For the local system, we consider the generating series%
\footnote{derivatives $\Phi_{h}^{(k)}$ will be taken with respect to $\tau\,(=\ell(q))$%
} $\Phi_{h}(q):=\frac{1}{(2\pi i)^{3}}\sum_{d\geq1}N_{d}q^{d}$ of
the genus-zero Gromov-Witten invariants of $X^{\circ}$, and define
the small quantum product on $V$ by $e_{2}*e_{2}:=-(m+\Phi_{h}'''(q))e_{1}$
and $e_{i}*e_{j}:=e_{i}\cup e_{j}$ for $(i,j)\ne(2,2)$. This gives
rise to the Dubrovin connection
\[
\nabla:=\text{id}_{V}\otimes d+(e_{2}*)\otimes d\tau,
\]
which we view as a map from $\mathcal{V}\cong V\otimes\mathcal{O}_{\Delta^{*}}\to V\otimes\Omega_{\Delta^{*}}^{1}\cong\mathcal{V}\otimes\Omega_{\Delta^{*}}^{1}$,
and the $\CC$-local system $\mathbb{V}_{\mathbb{C}}:=\ker(\nabla)\subset\mathcal{V}$.

Now define a map $\tilde{\sigma}:V\to V\otimes\mathcal{O}(\Delta)$
by
\[
\tilde{\sigma}(e_{0}):=e_{0},\;\tilde{\sigma}(e_{1}):=e_{1},\;\tilde{\sigma}(e_{2}):=e_{2}+\Phi_{h}''e_{1}+\Phi_{h}'e_{0},
\]
\[
\tilde{\sigma}(e_{3}):=e_{3}+\Phi_{h}'e_{1}+2\Phi_{h}e_{0}.
\]
For any $\alpha\in V$, one easily checks that
\[
\sigma(\alpha):=\tilde{\sigma}\left(e^{-\tau[H]}\cup\alpha\right):=\sum_{k\geq0}\frac{(-1)^{k}\tau^{k}}{k!}\tilde{\sigma}\left([H]^{k}\cup\alpha\right)
\]
satisfies $\nabla\sigma(\alpha)=0$, hence yields an isomorphism $\sigma:V\overset{\cong}{\to}\Gamma(\mathfrak{H},\rho^{*}\mathbb{V}_{\CC})$
(where $\rho:\mathfrak{H}\to\Delta^{*}$ sends $\tau\mapsto q$).
Writing%
\footnote{cf. $\S1$ of \cite{DK} for the more general definition of $\hat{\Gamma}(X^{\circ})$%
} 
\[
\hat{\Gamma}(X^{\circ}):=\exp\left(-\frac{1}{24}{ch}_{2}(X^{\circ})-\frac{2\zeta(3)}{(2\pi i)^{3}}{ch}_{3}(X^{\circ})\right)\in V,
\]
the image of 
\[
\begin{array}{cccc}
\gamma: & K_{0}^{num}(X^{\circ}) & \longrightarrow & \Gamma(\mathfrak{H},\rho^{*}\mathbb{V}_{\CC})\\
 & \xi & \mapsto & \sigma(\hat{\Gamma}(X^{\circ})\cup{ch}(\xi))
\end{array}
\]
defines Iritani's $\ZZ$-local system $\mathbb{V}$ underlying $\mathbb{V}_{\CC}$.
The filtration $W_{\bullet}:=W(N)_{\bullet}$ associated to its monodromy
$T(\gamma(\xi))=\gamma(\mathcal{O}(-H)\otimes\xi)$ satisfies $W_{k}\mathcal{V}_{e}=\left(\oplus_{j\geq3-k/2}H^{j,j}\right)\otimes\mathcal{O}_{\Delta}.$

In order to compute the limiting period matrix of this $\ZZ$-VHS
over $\Delta^{*}$, we shall require a (multivalued) basis $\{\gamma_{i}\}_{i=0}^{3}$
of $\mathbb{V}$ satisfying $\gamma_{i}\in W_{2i}\cap\mathbb{V}$,
$\gamma_{i}\equiv e_{i}\text{ mod }W_{2i-2}$, and $[Q]_{\gamma}=[Q]_{e}$.
The corresponding $\QQ$-basis of $\tilde{\mathbb{V}}|_{q=0}=:V_{lim}$
is given by $\gamma_{i}^{lim}:=\tilde{\gamma}_{i}(0)$ where $\tilde{\gamma}_{i}:=e^{-\tau N}\gamma_{i}\in\Gamma(\Delta,\tilde{\mathbb{V}})$.
Of course, the $e_{i}$ are another basis of $V_{lim,\CC}$, and $\Omega_{lim}={}_{\gamma^{lim}}[\text{id}]_{e}$.
Note that since $N_{lim}=-(2\pi i)Res_{q=0}(\nabla)=-(e_{2}*)|_{q=0}=-(e_{2}\cup)|_{q=0},$
we have 
\[
[N_{lim}]_{e}=\left(\begin{array}{cccc}
0 & 0 & 0 & 0\\
-1 & 0 & 0 & 0\\
0 & m & 0 & 0\\
0 & 0 & 1 & 0
\end{array}\right).
\]

A basis of the form we require is obtained by considering the Mukai
pairing
\[
\langle\xi,\xi'\rangle:=\int_{X^{\circ}}ch(\xi^{\vee}\otimes\xi')\cup Td(X^{\circ})
\]
on $K_{0}^{num}(X^{\circ})$. Since $\langle\xi,\xi'\rangle=Q(\gamma(\xi),\gamma(\xi'))$,
any Mukai-symplectic%
\footnote{That is, $\langle\xi_{i},\xi_{3-j}\rangle=0$ unless $i=j$, in which
case it is $+1$ for $i=0,1$ and $-1$ for $i=2,3$.%
} basis of $K_{0}^{num}(X^{\circ})$ of the form \begin{equation}\label{e:4*}\begin{aligned} & \xi_1 = \mathcal{O} + A\mathcal{O}_H + B\mathcal{O}_L + C\mathcal{O}_p  \\ & \xi_2 = \mathcal{O}_H + D\mathcal{O}_L + E\mathcal{O}_p  \\ & \xi_3 = -\mathcal{O}_L + F\mathcal{O}_p \\ & \xi_4 = \mathcal{O}_p \end{aligned}\end{equation}will
produce $\gamma_{i}:=\gamma(\xi_{i})$ satisfying the above hypotheses.
In this case, taking
\[
\sigma_{\infty}(\alpha):=\lim_{q\to0}\tilde{\sigma}(\alpha),\;\;\gamma_{\infty}(\xi):=\sigma_{\infty}\left(\hat{\Gamma}(X^{\circ})\cup{ch}(\xi)\right),
\]
we have $\gamma_{i}^{lim}=\gamma_{\infty}(\xi_{i})$.

We now run this computation. Let $c(X^{\circ})=1+a[L]+b[p]$ be the
Chern class of $X^{\circ}$; note that there is no $[H]$ term due
to the fact that $X^{\circ}$ is Calabi-Yau. Since the Chern character
is ${ch}(X^{\circ})=3-a[L]+\frac{b}{2}[p]$ and the Todd class is
$Td(X^{\circ})=1+\frac{a}{12}[L]$, $\hat{\Gamma}(X^{\circ})=1+\frac{a}{24}[L]-\frac{b\zeta(3)}{(2\pi i)^{3}}[p].$
This yields\begin{equation*}\Small\begin{aligned} & \gamma^{lim}_3 = e_3 + Ae_2 + \left( -B + \frac{m}{2}A - \frac{a}{24}\right) e_1  + \left(C - B +  \frac{4m+a}{24}A -b\frac{\zeta(3)}{(2\pi i)^3}\right) e_0 \\ &  \gamma^{lim}_2 = e_2 + \left( -D + \frac{m}{2} \right) e_1  + \left( E - D + \frac{4m+a}{24} \right) e_0  \\ &  \gamma^{lim}_1 = e_1  + (F+1)e_0  \\ &  \gamma^{lim}_0 = e_0  \end{aligned}\end{equation*}Imposing
the symplectic condition produces constraints $1+F+A=0$ and $\frac{a+2m}{12}-D+E-AD+B=0$.
After normalizing%
\footnote{$A=0$ is the canonical normalization of the local coordinate; the
remaining choices are made to simplify the end result.%
} $A=B=C=D=0$ ($\implies F=-1,\, E=-\frac{a+2m}{12}$) in \eqref{e:4*},
expressing each $e_{i}$ in terms of $\{\gamma_{i}^{lim}\}$ gives
the columns of \begin{equation}\Omega_{lim} =  \left( \begin{array}{cccc} 1 & 0 & 0 & 0 \\ 0 & 1 & 0 & 0 \\ \frac{a}{24} & -\frac{m}{2}& 1 & 0 \\ \frac{b\zeta(3)}{(2\pi i)^3} & \frac{a}{24} & 0 & 1 \end{array} \right) .\end{equation}

To compute $N$ (with these normalizations), we apply $\mathcal{O}(-H)\otimes$
to the $\xi_{i}$ in $K_{0}^{num}(X^{\circ})$; then 
\[
[T]_{\gamma}=\left[\mathcal{O}(-H)\otimes\right]_{\xi}=\left(\begin{array}{cccc}
1 & 0 & 0 & 0\\
-1 & 1 & 0 & 0\\
0 & m & 1 & 0\\
-\frac{a+2m}{12} & m & 1 & 1
\end{array}\right),
\]
whereupon taking $\log$ gives 
\[
\left[N_{lim}\right]_{\gamma^{lim}}=[N]_{\gamma}=\left(\begin{array}{cccc}
0 & 0 & 0 & 0\\
-1 & 0 & 0 & 0\\
\frac{m}{2} & m & 0 & 0\\
-\frac{a}{12} & \frac{m}{2} &  & 0
\end{array}\right).
\]
The data required to compute $N$ and $\Omega_{lim}$ for the complete
intersection Calabi-Yau (CICY) examples from \cite{DM} is displayed in the 
table below. Here for example ``$\PP^{5}[3,3]$'' means that $X^{\circ}$ is a 
complete
intersection of bidegree $(3,3)$ in $\PP^{5}$. \begin{table}[h]\centering \begin{tabular}{|l|c|c|c|} \hline $X^\circ$ & m & a & b  \\ \hline \hline $\mathbb{P}^{4}[5]$ & 5 & 50 & -200  \\ \hline $\mathbb{P}^{5}[2,4]$ & 8 &56 & -176  \\ \hline $\mathbb{P}^{5}[3,3]$ & 9 & 54 & -144  \\ \hline $\mathbb{P}^{6}[2,2,3]$ &12 & 60 & -144  \\ \hline $\mathbb{P}^{7}[2,2,2,2]$ & 8 & 64 & -128  \\ \hline $\mathbb{WP}_{1,1,1,2,5}^{4}[10]$& 10 & 340 & -2880  \\ \hline $\mathbb{WP}_{1,1,1,1,4}^{4}[8]$ & 8 & 176 & -1184  \\ \hline $\mathbb{WP}_{1,1,2,2,3,3}^{5}[6,6]$ &36 & 792 & -4320  \\ \hline $\mathbb{WP}_{1,1,1,2,2,3}^{5}[4,6]$ &24 & 384 & -1872  \\ \hline $\mathbb{WP}_{1,1,1,1,2}^{4}[6]$ &6 & 84 & -408  \\ \hline $\mathbb{WP}_{1,1,1,1,1,3}^{5}[2,6]$ &12 & 156 & -768  \\ \hline $\mathbb{WP}_{1,1,1,1,2,2}^{5}[4,4]$ &16 & 160 & -576  \\ \hline $\mathbb{WP}_{1,1,1,1,1,2}^{5}[3,4]$ &12 & 96 & -312  \\ \hline
\end{tabular}  \end{table} 

\noindent Since $X^{\circ}$ is smooth, the Chern numbers may be calculated using $$c(X^{\circ})=\frac{c(\mathbb{P})|_{X^{\circ}}}{c(N_{X^{\circ}/\mathbb{P}})}=\frac{\prod_{i=0}^{3+r} (1+\delta_i [H])}{\prod_{k=1}^r (1+d_k [H]).}$$

\begin{rem}
\label{rem 4.1}An interesting case not included amongst the CICY
examples is the so called ``14th case VHS'', labeled ``${\bf I}$''
in {[}loc. cit.{]}. It is shown in \cite{CDLNT} that this VHS arises
from the $Gr_{3}^{W}H^{3}$ of a subfamily contained in the singular
locus of a larger family of hypersurfaces in weighted-projective space.
The LMHS of this sort of example is probably inaccessible to the above
approach. The technique of the next section provides a possible approach
to such examples.
\end{rem}

\section{Calabi-Yau variations from middle convolutions}

Middle convolution is a binary operation on local systems introduced
by Katz \cite{Ka} to study the construction of rigid local systems
on Zariski open subsets $U\subset\PP^{1}$. Recent work of Dettweiler
and others (e.g. \cite{D,DR,DS}) has demonstrated the Hodge-theoretic
importance of this construction, of which we shall give only the briefest
description. The main point for us is that it yields interesting 
Calabi-Yau type variation for which the limiting invariant $\xi$
above may be computed \emph{directly}. In this way we can see where
the rational multiples of $\zeta(3)$ (or $\zeta(5)$) come from,
in contrast to the approach of the last section.

\subsection{The variations}

If $\{\underline a\}$ and $\{\underline b\}$ are finite sets of points
in $\mathbb A^1$ we define $\{\underline c\} = \{\underline a\}\ast\{\underline b\}$ to 
be the set obtained by taking all sums of pairs $a_j + b_k$
from $\{\underline a\}$ and $\{\underline b\}$.
Let $U_{1}=\mathbb{A}_{x}^{1}\backslash\{a_{1},\ldots,a_{m}\}$,
$U_{2}=\mathbb{A}_{z}^{1}\backslash\{b_{1},\ldots,b_{n}\}$, and $U_{3}=\mathbb{A}_{y}^{1}\backslash\{c_{1},\ldots,c_{p}\}$. Let $U\subset\mathbb{A}_{(x,y)}^{2}$ be the
Zariski open where $\prod_{j}(x-a_{j})\prod_{k}((y-x)-b_{k})\prod_{l}(y-c_{l})$
does not vanish. We have a diagram\[\xymatrix{ U_3 \times \PP^1 \ar @{->>} [d]_{\bar{\pi}_3} & & U_2 \\ U_3 & U \ar [l]^{\pi_3} \ar [d]_{\pi_1} \ar [ru]_{\pi_2} \ar @{_(->} [lu]_{\jmath} \\ & U_1} \]where
$\pi_{1}(x,y):=x$, $\pi_{2}(x,y):=y-x$, and $\pi_{3}(x,y)=y$. Given
local systems $\VV_{i}\to U_{i}$ ($i=1,2$), their \emph{middle convolution}
is the local system on $U_{3}$ defined by
\[
\VV_{1}*\VV_{2}:=R^{1}(\bar{\pi}_{3})_{*}\left(\jmath_{*}\left(\pi_{1}^{*}\VV_{1}\otimes\pi_{2}^{*}\VV_{2}\right)\right).
\]

Now suppose (following \cite[sec. 2.6]{D}) that the local systems
are motivic, say $\VV_{i}=Gr_{d_{i}}^{W}P_{i}R^{d_{i}}(\rho_{i})_{!}\QQ_{Y_{i}}$
($i=1,2$), where $Y_{i}\overset{\rho_{i}}{\to}U_{i}$ are smooth
morphisms and $P_{i}\in\QQ[Aut(\rho_{i})]$ idempotents. The situation
is described by the diagram\[\xymatrix{ & & U_2 & Y_2 \ar [l]_{\rho_2} \ar @(ur,dr)[]^{G_2} \\ U_3 & U \ar [ru]^{\pi_2} \ar [l] \ar [d]_{\pi_1}  & Y_1 \boxtimes Y_2 : \ar [l] \ar @/_1pc/ [ll]_{\rho_3} \ar @{=} [r] & \pi_1^* Y_1 \times_U \pi_2^*Y_2  \ \\ & U_1 & Y_1 \ar [l]_{\rho_1} \ar @(ur,dr)[]^{G_1}  }\]and
the middle convolution is described by
\[
\VV_{1}*\VV_{2}=Gr_{d_{1}+d_{2}+1}^{W}(P_{1}\times P_{2})R^{d_{1}+d_{2}+1}(\rho_{3})_{!}\QQ_{Y_{1}\boxtimes Y_{2}}.
\]

By iteratively alternating this construction with quadratic twists
as described in \cite[sec. 2.3-4]{DR}, we obtain a sequence of VHS
$\cV_{d}$ over $\PP^{1}\backslash\{0,1,\infty\}$ of weight $d$,
with $h^{d,0}=1$, for each $d\in\mathbb{N}$. From the motivic perspective,
for each iteration we begin with a family $\cX_{d-1}=\cup_{t\in\PP^{1}\backslash\{0,1,\infty\}}X_{d-1}(t)$
of ``singular Calabi-Yau'' $(d-1)$-folds (with involution $\sigma_{1}$)
over $U_{1}=\PP^{1}\backslash\{0,1,\infty\}$, and the $\sqrt{z}$-double-cover
$Y_{2}$ (with involution $\sigma_{2}$) over $U_{2}=\PP^{1}\backslash\{0,\infty\}$.
Taking $Y_{1}$ to be a quadratic twist (by $\sqrt{t}$ or $\sqrt{1-t}$)
of $\cX_{d-1}\backslash(\cX_{d-1})^{\sigma_{1}}$, we then apply to
$Y_{1}\boxtimes Y_{2}$ the ``projector'' of quotienting by $\sigma_{1}\times\sigma_{2}$,
producing $Y_{3}$. This has a natural compactification to a family
$\cX_{d}$ of ``singular C-Y'' $d$-folds with involution $\sigma_{3}$
over $U_{3}$, in which $Y_{3}=\cX_{d}\backslash(\cX_{d})^{\sigma_{3}}=:\cX_{d}^{-}$.
The local system underlying $\cV_{d}$ is just the $\sigma_{3}$-anti-invariants
in $Gr_{d}^{W}R^{d}(\rho_{3})_{!}\QQ_{Y_{3}}$.

The $\{X_{d}(t)\}$ produced by this algorithm (which are singular
for $d\geq2$) all take the form $w^{2}=f_{d}(x_{1},\ldots,x_{d},t)$,
and include:%
\footnote{Note that our parameter $t$ is inverse to that in \cite{DR}; for
odd $d$, we have also removed a final quadratic twist present in
{[}op. cit.{]} (to rid $f_{d}$ of a factor of $(1-t)$). %
}
\begin{lyxlist}{00.00.0000}
\item [{$d=1$:}] $w^{2}=(1-tx)x(x-1)$\\
(Legendre elliptic curve)
\item [{$d=3$:}] $w^{2}=(1-tx_{3})x_{3}(x_{2}-x_{3})(x_{2}-1)(x_{1}-x_{2})(x_{1}-1)x_{1}$\\
(CY 3-fold family, cf. Remark \ref{rem 5.2})
\item [{$d=6$:}] $w^{2}=\begin{array}[t]{c}
(1-tx_{6})(1-x_{6})(x_{5}-x_{6})x_{5}(x_{4}-x_{5})(1-x_{4})\times\\
(x_{3}-x_{4})x_{3}(x_{2}-x_{3})(1-x_{2})(x_{1}-x_{2})x_{1}(1-x_{1}).
\end{array}$
\end{lyxlist}
Each has an obvious involution $\sigma$ given by $w\mapsto-w$. Write
$\pi_{d}:\mathcal{X}_{d}^{-}\to\PP^{1}\backslash\{0,1,\infty\}$ so
that $\VV_{d}=\left(Gr_{d}^{W}R^{d}(\pi_{d})_{!}\QQ_{\cX_{d}^{-}}\right)^{-\sigma}$.
\begin{prop}
\label{prop51}(i) For $1\leq d\leq6$, $\cV_{d}$ is a VHS of weight
$d$ and rank $d+1$, with Hodge numbers all $1$.

(ii) \cite{DR} $\cV_{6}$ has MTG $G_{2}$.\end{prop}
\begin{proof}
How (ii) follows from the results of \cite{DR} is explained in \cite[sec. 9]{KP},
while (i) follows from the proof of Theorem 1.3.1 in \cite{DR}. In
particular, the table in that proof (with $0$ and $\infty$ swapped,
as our $1/t$ is their $\PP^{1}$ parameter) shows that the monodromy
at $t=0$ is a single Jordan block $U(d+1)$, which can only happen
for rank $d+1$ if the Hodge numbers are $(1,1,\ldots,1)$.\end{proof}
\begin{rem}
\label{rem 5.2}The $\{X_{3}(t)\}$ are degree-8 hypersurfaces in
$\mathbb{WP}(1,1,1,1,4)$, which are C-Y 3-folds after desingularization
(or for purposes of computing $(Gr_{3}^{W}R^{3}(\pi_{3})_{!}\QQ_{\cX_{3}^{-}})^{-\sigma}$).
Note that this is \emph{not} the mirror family for which the LMHS
was computed in $\S4$. Its LMHS at $t=0$ does not appear to be accessible
by mirror symmetry, since it belongs to the singular locus of a much
larger variation, and does not meet the large complex structure limit
of this larger family. We also note that while for $d=1,2$ the vanishing
cycle period $\int_{\mu_{t}}\omega_{t}$ ($\S5.2$) is a hypergeometric
function (up to quadratic twist), for $d\geq3$ this is not the case.
So then methods of computing LMHS using Meijer $G$-functions \cite{GL}
would also not be applicable.\end{rem}
\begin{rem}Referring to \cite[p. 940]{DR} and accounting for the inversion and
quadratic twists, the monodromies of $\mathbb{V}_{d}$ are displayed
in the following table, \begin{table}[h]\centering \begin{tabular}{|l|c|c|c|} \hline  & $\text{at }0$ & $\text{at }1$ & $\text{at }\infty $ \\ \hline \hline $d=1$ & $U(2)$ & $U(2)$ & $-U(2)$  \\ \hline $d=3$ & $U(4)$ & $-U(2)\oplus \mathbf{1}^{\oplus 2}$ & $(-U(2))^{\oplus 2}$  \\ \hline $d=6$ & $U(7)$ & $U(2)^{\oplus 2}\oplus U(3)$ & $(-\mathbf{1})^{\oplus 4}\oplus \mathbf{1}^{\oplus 3}$ \\ \hline
\end{tabular}  \end{table} in which $U(n)$ denotes a Jordan block of rank $n$.

For the stalks, we have (writing $D_{t}:=X_{d}(t)^{\sigma}$, $X_{d}^{-}(t):=X_{d}(t)\backslash D_{t}$)
\[
\mathbb{V}_{d,t}\cong Gr_{d}^{W}H_{c}^{d}(X_{d}^{-}(t),\QQ)^{-\sigma}\cong Gr_{d}^{W}H_{d}(X_{d}^{-}(t),\QQ(-d))^{-\sigma},
\]
\[
\mathbb{V}_{d,t}^{\vee}\cong Gr_{-d}^{W}H_{d}(X_{d}(t),D_{t};\QQ)^{-\sigma}\cong Gr_{-d}^{W}H_{d}(X_{d}(t),\QQ)^{-\sigma}.
\]
On each $X_{d}(t)$ there are obvious $\sigma$-anti-invariant topological
$d$-cycles consisting of two sheets ($\pm w$) bounding on components
of $D_{t}$ (e.g. $\mu_{t}$ and $\tau_{t}$ below); clearly such
cycles span $(Gr_{-d}^{W})H_{d}(X_{d}(t),\QQ)^{-\sigma}$. By a topological
argument (omitted here), these may be moved off $D_{t}$, hence belong
to the image of $H_{d}(X_{d}^{-}(t))\to H_{d}(X_{d}(t))$. The resulting
isomorphism $\mathbb{V}_{d,t}\overset{\cong}{\to}\mathbb{V}_{d,t}^{\vee}(-d)$
allows us to pair homology cycles in $Gr_{-d}^{W}H_{d}(X_{d}(t),\QQ)^{-\sigma}$
and write classes in $\mathbb{V}_{d,t}\otimes\CC$ in terms of them.
That is, in a sense we may work as if $X_{d}(t)$ were smooth, which
is immensely convenient for the computations that follow. 
\end{rem}

\subsection{Cauchy residue method}

In each case $(d=1,3,6$), we are after the LMHS at $t=0$. The idea
is to compute $\cF^{d}(\cV_{d})_{nilp}$ by Cauchy residue. 

More precisely, assume $0<t\ll1$ and write $\omega_{t}=\frac{2^{d-1}}{(2\pi i)^{d}}\frac{dx_{1}\wedge\cdots\wedge dx_{d}}{w}\in\Omega^{d}(X_{d}(t))$
for the ``holomorphic form''%
\footnote{The notation means that $\omega_{t}$ pulls back to a holomorphic
form on a desingularization of $X_{d}(t)$; in particular, it gives
a class in $F^{d}Gr_{d}^{W}H_{c}^{d}(X_{d}^{-}(t),\CC)^{-\sigma}$.%
} and 
\[
\tau_{t}=\{(w,x_{1},\ldots,x_{d})\in X_{d}(\RR)\,|\,1\leq x_{1}\leq x_{2}\leq\cdots\leq x_{d}\leq t^{-1}\}
\]
for a family of cycles (with two branches coming from $\pm w$). Note
that there exists a family $\mu_{t}:=\{(w,x_{1},\ldots,x_{d})\in X_{d}(\RR)\,|\,\frac{1}{t}\leq x_{d}\leq\cdots\leq x_{1}<\infty\}$
($0<t\ll1$) of vanishing cycles with $(\mu_{t},\tau_{t})=1$ and
$\int_{\mu_{t}}\omega_{t}\to1$ as $t\to0$ (for example, using $\int_{1}^{\infty}\frac{du}{u\sqrt{u-1}}=\pi$
and the residue approach below). Hence $\tau_{t}$ and $\mu_{t}$
are correctly normalized; that is, they are the extremal members $\gamma_{d}$
resp. $\gamma_{0}$ of an integral symplectic basis $\{\gamma_{j}\}_{j=0}^{d}$
of $\VV_{d}$ over a punctured disk, in which the monodromy about
$t=0$ takes the form%
\footnote{The ``$\pm$'' is $(-1)^{d}$ for $d>1$.%
} \begin{equation}\label{e:T}[T]_{\gamma}=\left(\begin{array}{cccc}1 & 0 & \cdots & 0\\a & 1 & \ddots & \vdots\\\vdots & \ddots & \ddots & 0\\* & \cdots & \pm a & 1\end{array}\right).\end{equation}

Throughout this section, $\oint_{|t|=\epsilon}dt$ denotes integration
counterclockwise from $\arg(t)=-\pi$ to $\arg(t)=\pi$. Recalling
the notation $\ell(t)=\frac{\log(t)}{2\pi i}$, integration yields\begin{equation}\label{eqn a_jk}\Pi_{d}(t):=\int_{\tau_{t}}\omega_{t}=(-1)^{d}\sum_{j=0}^{d}\ell^{j}(t)\sum_{k\geq0}a_{jk}t^{k}\end{equation}whereupon\begin{equation}\label{eqn Cauchy}\frac{1}{2\pi i}\oint_{|t|=\epsilon}\frac{dt}{t}\int_{\tau_{t}}\omega_{t}=\underset{=:\Pi_{d}^{nilp}(\epsilon)}{\underbrace{(-1)^d \sum_{j=0}^{d}a_{j0}\ell^{j}(\epsilon)}}+\underset{\tiny \begin{matrix} \to 0 \\ \text{ with }\epsilon\end{matrix}}{\underbrace{\mathcal{O}(\epsilon\log^{d}\epsilon)}}.\end{equation}In
the remainder of this paper, ``$\equiv$'' shall be used to denote
working modulo $\mathcal{O}(\epsilon\log^{d}\epsilon)$.

If $\omega_{t}^{nilp}$ is the section of $\cF^{d}(\cV_{d})_{nilp}$
with period $1$ against $\mu_{t}$, then $\Pi_{d}^{nilp}(\epsilon)$
in \eqref{eqn Cauchy} is its period at $t=\epsilon$ against $\tau_{\epsilon}$.
Its full period vector takes the form \begin{equation}\label{eqn omega-nilp}[\omega_{\epsilon}^{nilp}]_{\gamma}=\left(\begin{array}{c}1\\a_{10}^{(d-1)}\ell(\epsilon)+a_{00}^{(d-1)}\\\vdots\\\sum_{j=0}^{d-1}a_{j0}^{(1)}\ell^{j}(\epsilon)\\\sum_{j=0}^{d}a_{j0}\ell^{j}(\epsilon)\end{array}\right).\end{equation}Applying
$e^{\ell(\epsilon)N}$ to this must yield\begin{equation}\label{eqn Omega_lim unn}{}^{t}\left(1,a_{00}^{(d-1)},\ldots,a_{00}^{(1)},a_{00}\right) ,\end{equation}from
which we deduce that $\mp a_{00}^{(1)}a=a_{10}$ and $a_{10}^{(d-1)}=-a$. 

So \eqref{eqn Omega_lim unn} is the first column of $\Omega_{lim}$
prior to canonically normalizing the local coordinate. To carry out
this normalization, we make the substitution $t=\alpha s$ in \eqref{eqn a_jk},
where $\ell(\alpha)=-\frac{a_{00}^{(d-1)}}{a_{10}^{(d-1)}}=\frac{a_{00}^{(d-1)}}{a}$,
and rewrite the right-hand side in powers of $s$ and $\ell(s)$.
Writing $\tilde{a}_{jk}$ (more generally $\tilde{a}_{jk}^{(i)}$)
for the modified coefficients and $\tilde{\Pi}_{d}^{nilp}(\epsilon):=(-1)^{d}\sum_{j=0}^{d}\tilde{a}_{j0}\ell^{j}(\epsilon)$
(more generally $\tilde{\omega}_{\epsilon}^{nlp}$) for the modified
periods, we repeat the above computation with the result that \eqref{eqn Omega_lim unn}
is replaced by
\[
^{t}\left(1,\tilde{a}_{00}^{(d-1)},\ldots,\tilde{a}_{00}^{(1)},\tilde{a}_{00}\right)={}^{t}\left(1,0,\ldots,\mp\frac{\tilde{a}_{10}}{a},\tilde{a}_{00}\right)
\]
which is now the correct first column of $\Omega_{lim}$. (Equivalently,
apply $e^{\ell(\alpha)N}$ to \eqref{eqn Omega_lim unn}.) In particular,
the extension class $\xi\in\CC/\QQ$ from the end of $\S$3 is given
by $\tilde{a}_{00}$ (for $d=3$) or $-\tilde{a}_{10}/a$ (for $d=6$).
More information is contained in the following
\begin{prop}
\label{prop ***}The bottom row of the normalized $\Omega_{lim}$
is 
\[
\left(\tilde{a}_{00},-\frac{1!\tilde{a}_{10}}{2},\frac{2!\tilde{a}_{20}}{2^{2}},\ldots,\frac{(d-1)!\tilde{a}_{d-1,0}}{(-2)^{d-1}},1\right).
\]
Moreover, we have $a=2$ and $\alpha=4^{d+1}$.\end{prop}
\begin{proof}
From the calculations we carry out for $d=1,3$ and $6$ in $\S\S$5.3-5 below,
it is evident that $(\tilde{a}_{d0}=)\, a_{d0}=\frac{2^{d}}{d!}$
and $a_{d-1,0}=-\frac{2^{d}(d+1)}{(d-1)!}\ell(4)$. In order for $\omega_{t}^{nilp}$
and its derivatives $\nabla_{t\frac{\partial}{\partial t}}\omega_{t}^{nilp}$
to be single-valued, we must have $(\tilde{a}_{k0}^{(d-k)}=)\, a_{k0}^{(d-k)}=\frac{(-a)^{k}}{k!}$
(for $1\leq k<d$) and $a_{d0}=\frac{a^{d}}{d!}$; so $a=2$.

By Griffiths transversality, the columns of the normalized $\Omega_{lim}$
are given by $\left\{ \frac{(2\pi i)^{k}}{k!a_{k0}^{(d-k)}}e^{\ell(t)N}\left(t\frac{\partial}{\partial t}\right)^{k}[\tilde{\omega}_{t}^{nilp}]_{\gamma}\right\} _{k=0}^{d}.$
In particular, the bottom row has entries $\frac{\tilde{a}_{k0}}{a_{k0}^{(d-k)}}=\frac{k!\tilde{a}_{k0}}{(-2)^{k}}$
($1\leq k<d$). Since normalization kills the $k=d-1$ entry, we must
have $\tilde{a}_{d-1,0}=0$. In order for replacing $\ell(t)$ by
$\ell(s)+\ell(\alpha)$ to eliminate the $\ell^{d-1}(t)$ term of
$\Pi_{d}^{nilp}(t)$, we need $\ell(\alpha)=-\frac{a_{d-1,0}}{a_{d0}d}=(d+1)\ell(4).$
\end{proof}
The main point is that $\Pi_{d}^{nilp}$ contains all the information
in $\Omega_{lim}$, and the normalization can be carried out using
$\Pi_{d}^{nilp}$alone: one just makes the substitution that kills
the $\ell^{d-1}(t)$ term. In our computations, this will simply mean
replacing $\Pi_{d}^{nilp}(\epsilon)$ by $\tilde{\Pi}_{d}^{nilp}(\epsilon)=\Pi_{d}^{nilp}(4^{d+1}\epsilon).$

\subsection{Computing the extension classes}

For $d=1$, the (normalized) $\tilde{a}_{00}$ is zero, but Conjecture
\ref{conj ggk} still has content: it says that the \emph{unnormalized}
$a_{00}\in\CC/\QQ$ should be (a rational multiple of) $\ell(q)$
for some $q\in\QQ^{*}$. While the conjecture is known for elliptic
curves (cf. \cite[(III.B.11)]{GGK}), checking it gives an initial
feasibility test for our Cauchy residue approach to $\Omega_{lim}$,
and motivates what shall take place in higher dimension. Referring
to \eqref{eqn Cauchy}, $\frac{1}{2\pi i}\oint_{|t|=\epsilon}\frac{dt}{t}\int_{\tau_{t}}\omega_{t}=$
\[
=\frac{2}{(2\pi i)^{2}}\oint_{|t|=\epsilon}\frac{dt}{t}\int_{1}^{\frac{1}{t}}\frac{dx}{\sqrt{x(x-1)(1-tx)}}.
\]
Substituting $u=\frac{(x-1)t}{1-t}$ yields
\[
=\frac{2}{(2\pi i)^{2}}\oint_{|t|=\epsilon}\frac{dt}{t}\int_{0}^{1}\frac{du}{\sqrt{u(1-u)(u+(1-u)t)}}
\]
\[
=\frac{1}{\pi i}\int_{0}^{1}\left(\oint_{|t|=\epsilon}\frac{1}{\sqrt{u+(1-u)t}}\frac{dt}{2\pi it}\right)\frac{du}{\sqrt{u(1-u)}}.
\]
Let $\eta = \frac{\epsilon}{1+\epsilon}$.  Then, we have the power series
expansions
$$
        \frac{1}{\sqrt{u+(1-u)t}}
         = \frac{1}{\sqrt{u}}
           \sum_{m\geq0}t^{m}{-\frac{1}{2} \choose m}\frac{(1-u)^{m}}{u^{m}}
$$
valid for $u\in[\eta,1]$, and 
$$
    \frac{1}{\sqrt{u+(1-u)t}} 
     = \sum_{m\geq0} t^{-m-\frac{1}{2}}{-\frac{1}{2} \choose m}\frac{u^{m}}{(1-u)^{m+\frac{1}{2}}}
$$
valid for $u\in[0,\eta]$. In the
former expansion, $\oint$ annihilates all but the $m\geq0$ term;
for the latter, we use%
\footnote{Note: it is not correct to ``go twice around the circle'' and kill
$t^{-m-\frac{1}{2}}$. The problem is that $\int_{0}^{1}\frac{du}{\sqrt{u(1-u)(u+(1-u)t)}}$
only matches the analytic continuation of $\int_{\tau_{t}}\omega_{t}$
for $\arg(t)\in(-\pi,\pi)$.%
} \begin{equation}\label{e:halfpowers}\oint_{|t|=\epsilon}t^{-(m+\frac{1}{2})}\frac{dt}{2\pi it}=\frac{(-1)^{m}\epsilon^{-(m+\frac{1}{2})}}{\pi(m+\frac{1}{2})}.\end{equation}Altogether,
the above
\[
=\frac{1}{\pi i}\int_{\eta}^{1}\frac{du}{u\sqrt{1-u}}+\frac{1}{\pi^{2}i}\sum_{m\geq0}\frac{\left|\binom{-\frac{1}{2}}{m}\right|\epsilon^{-(m+\frac{1}{2})}}{(m+\frac{1}{2})}\int_{0}^{\eta}\frac{u^{m-\frac{1}{2}}}{(1-u)^{m+1}}du.
\]
Working modulo $\mathcal{O}(\epsilon\log\epsilon)$, this becomes\tiny
\[
\equiv\frac{1}{\pi i}\left\{ \int_{\eta}^{1}\frac{du}{u}+\sum_{k\geq1}\left|{-\frac{1}{2} \choose k}\right|\int_{\eta}^{1}u^{k-1}du\right\} +\frac{1}{\pi^{2}i}\sum_{m\geq0}\frac{\left|\binom{-\frac{1}{2}}{m}\right|\epsilon^{-(m+\frac{1}{2})}}{(m+\frac{1}{2})}\int_{0}^{\eta}u^{m-\frac{1}{2}}du
\]
\small
\[
\equiv-2\ell(\epsilon)+\underset{-a_{00}}{\underbrace{\frac{1}{\pi i}\left\{ \sum_{k\geq1}\frac{\left|{-\frac{1}{2} \choose k}\right|}{k}+\frac{1}{\pi}\sum_{m\geq0}\frac{\left|\binom{-\frac{1}{2}}{m}\right|}{(m+\frac{1}{2})^{2}}\right\} }}.
\]
\normalsize A short computation now shows that $a_{00}=\frac{-2}{2\pi i}\left\{ \log4+\log4\right\} =\ell(\frac{1}{4^{4}})$,
and certainly $\frac{1}{4^{4}}\in\QQ^{*}$.

Essentially the same thing happens in general: in computing $\Pi_{d}^{nilp}$,
one has to face (for example, generalizing the $\int_{\eta}^{1}$
integral above)\begin{equation}\label{eqn main int}\sum_{k_{1},\ldots,k_{d}\geq0}\left(\prod_{j}\left|{-\frac{1}{2} \choose k_{j}}\right|\right)\int_{[0,1]^d \cap \{u_1\cdots u_d>\epsilon\} }\left(\prod_{j}u_{j}^{k_{j}-1}du_{j}\right)\end{equation}and
hence (as a byproduct) the constants\begin{equation}\label{eqn gamma}\gamma_{n}:=\sum_{k\geq1}\frac{\left|{-\frac{1}{2} \choose k}\right|}{k^{n}}=\frac {1}{2} {}_{n+2}F_{n+1}\left(\left.\begin{array}{c}1,\ldots,1,\frac{3}{2}\\2,\ldots,2\end{array}\right|1\right)\end{equation}for
$n\geq1$. The key observation about them is that (while $\gamma_{1}=\log4$)
\[
\gamma_{n}=q_{n}\zeta(n)+\text{``degenerate'' terms},
\]
where $q_{2}=1$, $q_{3}=2$, $q_{4}=\frac{9}{4}$, $q_{5}=6$, $q_{6}=\frac{79}{16}$,
etc. The generalization of the $\int_{0}^{\eta}$ integral above is
more complicated, with \begin{equation}\label{eqn gammatilde}\tilde{\gamma}_{n}:=\frac{1}{\pi}\sum_{k\geq0}\frac{\left|\binom{-\frac{1}{2}}{k}\right|}{(k+\frac{1}{2})^{n+1}}=\frac{2^{n+1}}{\pi}{}_{n+2}F_{n+1}\left(\left. \begin{array}{c} \frac{1}{2},\ldots,\frac{1}{2}\\ \frac{3}{2},\ldots ,\frac{3}{2} \end{array}\right| 1\right)\end{equation}as
well as some very interesting multiple series appearing. See the Appendix
for evaluation and discussion of the $\gamma_{n}$ and $\tilde{\gamma}_{n}$.

\subsection{Computing the LMHS of $\mathcal{V}_{3}$}

For $d=3$, \eqref{eqn Cauchy} is $\frac{2^{3}}{(2\pi i)^{3}}$ times
\[
\frac{1}{2\pi i}\oint_{|t|=\epsilon}\frac{dt}{t}\int_{1}^{\frac{1}{t}}\int_{1}^{x_{3}}\int_{1}^{x_{2}}\frac{1}{\sqrt{f_{3}(x_{1},x_{2},x_{3},t)}}dx_{1}dx_{2}dx_{3}
\]
which upon substituting $\tilde{X}_{i}=\frac{(x_{i}-1)t}{1-t}$ becomes
\[
=\frac{1}{2\pi i}\oint\frac{dt}{t}\int_{0}^{1}\int_{0}^{\tilde{X}_{3}}\int_{0}^{\tilde{X}_{2}}\frac{\sqrt{1-t}}{\sqrt{\tilde{F}_{3}(\tilde{X}_{1},\tilde{X}_{2},\tilde{X}_{3},t)}}d\tilde{X}_{1}d\tilde{X}_{2}d\tilde{X}_{3},
\]
where\small 
\[
\tilde{F}_{3}(\tilde{X}_{1},\tilde{X}_{2},\tilde{X}_{3},t)=\prod_{i=1}^{3}(1-\tilde{X}_{i})\prod_{i=1}^{2}(\tilde{X}_{i}-\tilde{X}_{i+1})\prod_{i=1,3}\left\{ (1-\tilde{X}_{i})t+\tilde{X}_{i}\right\} .
\]
\normalsize Note that the region of integration is now independent
of $t$; moving the $\oint$ inside and performing the further substitutions
$\tilde{X}_{3}=X_{3}$, $\tilde{X}_{2}=X_{2}X_{3}$, $\tilde{X}_{1}=X_{1}X_{2}X_{3}$,
the above integral\begin{equation}\label{e:tint3}=\iiint_{[0,1]^{\times3}}\left(\oint_{|t|=\epsilon}\frac{\sqrt{1-t}}{\sqrt{F_{3}(X_{1},X_{2},X_{3},t)}}\frac{dt}{2\pi it}\right)dX_{1}dX_{2}dX_{3}\end{equation}where
$F(X_{1},X_{2},X_{3},t)=$
\[
\left\{ (1-X_{3})t+X_{3}\right\} \left\{ (1-X_{1}X_{2}X_{3})t+X_{1}X_{2}X_{3}\right\} X_{1}X_{2}\prod_{i=1}^{3}(1-X_{i}).
\]

Next we break $[0,1]^{\times3}$ in \eqref{e:tint3} into 4 regions
according to whether (I) $X_{1}X_{2}X_{3}>\eta:=\frac{\epsilon}{1+\epsilon}$,
(IIa) $X_{2}X_{3}>\eta>X_{1}X_{2}X_{3}$, (IIb) $X_{3}>\eta>X_{2}X_{3}$,
or (III) $\eta>X_{3}$. The expansion of $F_{3}(X_{1},X_{2},X_{3},t)^{-\frac{1}{2}}$
depends on the region:
\begin{lyxlist}{00.00.0000}
\item [{(I)}] $\sum_{a,b\geq0}t^{a+b}\binom{-\frac{1}{2}}{a}\binom{-\frac{1}{2}}{b}\frac{X_{1}^{-a}X_{2}^{-a}X_{3}^{-a-b}(1-X_{3})^{b-\frac{1}{2}}(1-X_{1}X_{2}X_{2})^{a}}{\sqrt{(1-X_{1})(1-X_{2})}}$,
\item [{(IIa,b)}] $\sum_{a,b\geq0}t^{a-b-\frac{1}{2}}\binom{-\frac{1}{2}}{a}\binom{-\frac{1}{2}}{b}\frac{X_{1}^{b}X_{2}^{b}X_{3}^{b-a}(1-X_{3})^{a-\frac{1}{2}}(1-X_{1}X_{2}X_{3})^{-b}}{\sqrt{(1-X_{1})(1-X_{2})}}$,
\item [{(III)}] $\sum_{a,b\ge0}t^{-a-b}\binom{-\frac{1}{2}}{a}\binom{-\frac{1}{2}}{b}\frac{X_{1}^{a}X_{2}^{a}X_{3}^{a+b}(1-X_{3})^{-a-\frac{1}{2}}(1-X_{1}X_{2}X_{3})^{-b}}{\sqrt{(1-X_{1})(1-X_{2})}}$.
\end{lyxlist}
For purposes of working modulo $\mathcal{O}(\epsilon\log^{3}\epsilon)$,
computation shows that we may replace $\sqrt{1-t}$ and $(1-X_{1}X_{2}X_{3})$
by $1$; whereas $(1-X_{i})^{-\frac{1}{2}}$ is always expanded as%
\footnote{the Pochhammer symbol $(\frac{1}{2})_{k}=\left|\binom{-\frac{1}{2}}{k}\right|$.%
} $\sum_{k_{i}\geq0}(\frac{1}{2})_{k_{i}}X_{i}^{k_{i}}$. (The special
case $(1-X_{3})^{a-\frac{1}{2}}$ is expanded when $a=0$ and replaced
by $1$ when $a>0$.) We may also replace $\eta$ by $\epsilon$ in
the triple integrals, which become (I) $\int_{\epsilon}^{1}\int_{\frac{\epsilon}{X_{3}}}^{1}\int_{\frac{\epsilon}{X_{2}X_{3}}}^{1}$,
(IIa) $\int_{\epsilon}^{1}\int_{\frac{\epsilon}{_{X_{3}}}}^{1}\int_{0}^{\frac{\epsilon}{X_{2}X_{3}}}$,
(IIb) $\int_{\epsilon}^{1}\int_{0}^{\frac{\epsilon}{X_{3}}}\int_{0}^{1}$,
and (III) $\int_{0}^{\epsilon}\int_{0}^{1}\int_{0}^{1}$. Region (III)
makes no contribution.

Performing the $\oint$ in region (I) kills all terms except $(a,b)=(0,0)$.
So the portion of \eqref{e:tint3} over region (I) is
\[
\int_{\epsilon}^{1}\int_{\frac{\epsilon}{X_{3}}}^{1}\int_{\frac{\epsilon}{X_{2}X_{3}}}^{1}\prod_{i=1}^{3}X_{i}^{-1}(1-X_{i})^{-\frac{1}{2}}dX_{i}=
\]
\[
\sum_{k_{1},k_{2},k_{3}\geq0}\left(\prod_{i=1}^{3}\left(\frac{1}{2}\right)_{k_{i}}\right)\int_{\epsilon}^{1}\int_{\frac{\epsilon}{X_{3}}}^{1}\int_{\frac{\epsilon}{X_{2}X_{3}}}^{1}\left(\prod_{j=1}^{3}X_{j}^{k_{j}-1}dX_{j}\right)
\]
which is now \eqref{eqn main int} with $d=3$. Repeatedly applying
the formula\begin{equation}\label{e:log}-\int_{\mu}^{1}(\log^{r}x)x^{k-1}dx=\left\{ \begin{array}{cc}\frac{1}{r+1}\log^{r+1}\mu, & k=0\\\sum_{\ell=0}^{r}\frac{(-1)^{\ell}\mu^{k}r!}{k^{\ell+1}(r-\ell)!}\log^{r-\ell}\mu+\frac{(-1)^{r+1}r!}{k^{r+1}}, & k\neq0\end{array}\right.\end{equation}and
throwing out terms with positive powers of $\epsilon$, we arrive
at \begin{equation}\label{e:sharp1}-\frac{1}{6}\log^{3}\epsilon+\frac{3}{2}\gamma_{1}\log^{2}\epsilon+(3\gamma_{2}-3\gamma_{1}^{2})\log\epsilon+(\gamma_{1}^{3}-6\gamma_{1}\gamma_{2}+3\gamma_{3}).\end{equation}For
region (IIa,b), applying \eqref{e:halfpowers} and computing the triple
integrals (and simplifying results using \eqref{e:gteval}) yields\begin{equation}\label{e:sharp2}\frac{\gamma_{1}}{2}\log^{2}\epsilon-(2\gamma_{1}^{2}+\beta)\log\epsilon+(\gamma_{1}^{3}-2\gamma_{1}\gamma_{2}+\gamma_{1}\beta+\nu-\psi)\end{equation}for
(IIa) and\begin{equation}\label{e:sharp3}-\delta\log(\epsilon)+(\gamma_1\delta+\nu')\end{equation}for
(IIb). The meaning of the assorted Greek letters is 
\[
\nu:=\frac{1}{\pi}\sum{}^{'}\frac{(\frac{1}{2})_{k_{1}}(\frac{1}{2})_{k_{2}}(\frac{1}{2})_{a}(\frac{1}{2})_{b}}{(b-a+\frac{1}{2})(b+k_{1}+\frac{1}{2})(a+k_{1})(a+k_{2})},
\]
\[
\nu':=\frac{1}{\pi}\sum{}^{'}\frac{(\frac{1}{2})_{k_{1}}(\frac{1}{2})_{k_{2}}(\frac{1}{2})_{a}(\frac{1}{2})_{b}}{(b-a+\frac{1}{2})(b+k_{1}+\frac{1}{2})(b+k_{2}+\frac{1}{2})(a+k_{1})},
\]
\[
\beta:=\frac{1}{\pi}\sum{}^{'}\frac{(\frac{1}{2})_{a}(\frac{1}{2})_{b}}{(b+\frac{1}{2})a(a+b+\frac{1}{2})},\;\;\;\;\delta:=\frac{1}{\pi}\sum{}^{'}\frac{(\frac{1}{2})_{a}(\frac{1}{2})_{b}}{(b+\frac{1}{2})^{2}(a+b+\frac{1}{2})},
\]
\[
\psi:=\frac{1}{\pi}\sum{}^{'}\frac{(\frac{1}{2})_{a}(\frac{1}{2})_{b}}{(b+\frac{1}{2})a^{2}(a+b+\frac{1}{2})},
\]
where $\sum{}^{'}$ denotes summation over all 2,3, or 4-tuples of
non-negative integers for which the denominator is nonzero.

Now it is easy to prove that $\beta+\delta=\gamma_{1}\tilde{\gamma}_{1}+\tilde{\gamma}_{2}=2\gamma_1^2 +\gamma_2$.
Adding \eqref{e:sharp1}, \eqref{e:sharp2}, and \eqref{e:sharp3},
replacing $\log\epsilon$ by $\log s+4\gamma_{1}$, and using \eqref{e:g1eval}-\eqref{e:g3eval}
gives
\[
-\frac{1}{6}\log^{3}s+2\zeta(2)\log s+\left(6\zeta(3)-2\gamma_{1}\zeta(2)-\frac{8}{3}\gamma_{1}^{3}+\nu+\nu'-\psi\right).
\]
The following will be proved in $\S5.5$:
\begin{lem}
\label{lem:g2=00003D}$\nu+\nu'-\psi=\frac{8}{3}\gamma_{1}^{3}+2\gamma_{1}\zeta(2)-12\zeta(3)$.
\end{lem}
Reinstating the factor of $\frac{2^{3}}{(2\pi i)^{3}}$, we have
\begin{thm}
\label{thm:d=00003D3} For $d=3$, the canonically
normalized $\tilde{\Pi}_{d}^{nilp}(s)$ is given by
\[
-\frac{4}{3}\ell^{3}(s)+16\frac{\zeta(2)}{(2\pi i)^{2}}\ell(s)\underset{\tilde{a}_{00}}{\underbrace{-48\frac{\zeta(3)}{(2\pi i)^{3}}}}.
\]

\end{thm}
We conclude that the extension class $\xi=\tilde{a}_{00}\in\CC/\QQ$
satisfies Conjecture 2.3.

\subsection{The $G_{2}$-VHS $\mathcal{V}_{6}$}

The comparable simplifications on \eqref{eqn Cauchy} for $d=6$ lead
to ($\frac{2^{6}}{(2\pi i)^{6}}$ times)\begin{equation}\label{e:5.5**}\int_{[0,1]^{\times6}}\left(\oint_{|t|=\epsilon}\frac{\sqrt{X_{5}X_{6}}\sqrt{1-t}}{\sqrt{F_{6}(X_{1},\ldots,X_{6},t)}}\frac{dt}{2\pi it}\right)dX_{1}\cdots dX_{6},\end{equation}where $F(X_1,\ldots,X_6,t)=$
\[
X_{1}X_{2}\prod_{i=1}^{6}(1-X_{i})\prod_{j=1,3,5}\left\{ (1-X_{j}\cdots X_{6})t+X_{j}\cdots X_{6}\right\} .
\]
The region of integration breaks as before into (I) $X_{1}\cdots X_{6}>\eta$,
(IIa,b) $X_{3}\cdots X_{6}>\eta>X_{1}\cdots X_{6}$, (IIIa,b) $X_{5}X_{6}>\eta>X_{3}\cdots X_{6}$,
and (IVa,b) $\eta>X_{5}X_{6}$. Working modulo $\mathcal{O}(\epsilon\log^{6}\epsilon)$,
the region (I) integral again is just \eqref{eqn main int}, and yields\small
\[
\frac{1}{720}\log^{6}\epsilon-\frac{\gamma_{1}}{20}\log^{5}\epsilon+\left\{ -\frac{\gamma_{2}}{4}+\frac{5\gamma_{1}^{2}}{8}\right\} \log^{4}\epsilon+\left\{ -\gamma_{3}+5\gamma_{1}\gamma_{2}-\frac{10}{3}\gamma_{1}^{3}\right\} \log^{3}\epsilon
\]
\[
+\left\{ -3\gamma_{4}+15\gamma_{1}\gamma_{3}+\frac{15}{2}\gamma_{2}^{2}-30\gamma_{1}^{2}\gamma_{2}+\frac{15}{2}\gamma_{1}^{4}\right\} \log^{2}\epsilon+\left\{ -6\gamma_{5}+30\gamma_{1}\gamma_{4}\right.
\]
\[
\left.+30\gamma_{2}\gamma_{3}-60\gamma_{1}\gamma_{2}^{2}-60\gamma_{1}^{2}\gamma_{3}+60\gamma_{1}^{3}\gamma_{2}-6\gamma_{1}^{5}\right\} \log\epsilon+\left\{ -6\gamma_{6}+30\gamma_{1}\gamma_{5}+30\gamma_{2}\gamma_{4}\right.
\]
\[
\left.+15\gamma_{3}^{2}-20\gamma_{2}^{3}-120\gamma_{1}\gamma_{2}\gamma_{3}-60\gamma_{1}^{2}\gamma_{4}+90\gamma_{1}^{2}\gamma_{2}^{2}+60\gamma_{1}^{3}\gamma_{3}-30\gamma_{1}^{4}\gamma_{2}+\gamma_{1}^{6}\right\} .
\]
\normalsize Notice (in view of \eqref{e:g5eval}) the $-36\zeta(5)$
in the coefficient of $\log\epsilon$.

Unfortunately, other regions produce some series we are (at present)
unable to evaluate: for instance, the coefficient of $\log\epsilon$
in the (IIb) integral contains the term

\tiny
\[
\frac{1}{\pi}\sum{}^{'}\frac{(\frac{1}{2})_{k_{1}}(\frac{1}{2})_{k_{2}}(\frac{1}{2})_{k_{3}}(\frac{1}{2})_{k_{4}}(\frac{1}{2})_{a_{1}}(\frac{1}{2})_{a_{2}}(\frac{1}{2})_{a_{3}}}{(a_{3}-a_{2}-a_{1}+\frac{1}{2})(k_{1}+a_{3}+\frac{1}{2})(k_{2}+a_{3}+\frac{1}{2})(k_{3}-k_{2}-a_{2})(k_{4}-k_{3})(k_{4}+a_{1})}.
\]
\normalsize So we will limit ourselves here to evaluating only the
first four terms of $\tilde{\Pi}_{6}^{nilp}$. Adding the contributions
from (IIa,b) to the first line of the region (I) result, gives
\[
\frac{1}{6!}\log^{6}\epsilon-\frac{7\gamma_{1}}{5!}\log^{5}\epsilon+\left\{ \frac{49}{48}\gamma_{1}^{2}-\frac{5}{24}\zeta(2)\right\} \log^4(\epsilon)
\]
\[
+\left\{ -\frac{109}{12}\gamma_{1}^{3}+\frac{37}{6}\gamma_{1}\zeta(2)-2\zeta(3)-\frac{\nu+\nu'-\psi}{6}\right\} \log^{3}\epsilon.
\]
Normalizing this and multiplying by $\frac{1}{(2\pi i)^{6}}$, modulo
$\mathcal{O}(\log^{2}s)$ we have $\frac{1}{2^{6}}\tilde{\Pi}_{6}^{nilp}(s)\equiv$\small
\[
\frac{1}{6!}\ell^{6}(s)\underset{\tilde{a}_{40}/2^{6}}{\underbrace{-\frac{5\zeta(2)}{24(2\pi i)^{2}}}}\ell^{4}(s)+\underset{\tilde{a}_{30}/2^{6}}{\underbrace{\frac{1}{(2\pi i)^{3}}\left(\frac{4}{9}\gamma_{1}^{3}+\frac{1}{3}\gamma_{1}\zeta(2)-2\zeta(3)-\frac{\nu+\nu'-\psi}{6}\right)}}\ell^{3}(s).
\]
\normalsize Referring to the end of $\S3$, Proposition \ref{prop ***}
now implies that $\tilde{a}_{40},\tilde{a}_{30}\in\QQ$. For $\tilde{a}_{40}$,
this is clearly the case; but $\tilde{a}_{30}$ belongs to $i\RR$
and so must be zero, proving Lemma \ref{lem:g2=00003D}. This is particularly
striking, as our knowledge that $\tilde{a}_{30}$ is rational depends
on Proposition \ref{prop51}(ii). So it is thanks to $G_{2}$ that
we can evaluate $\nu+\nu'-\psi$, and with it, the $d=3$ LMHS.

For $d=6$, we do expect Conjecture \ref{conj ggk} to remain true:
\begin{conjecture}
\label{conj:d=00003D6}The canonically normalized $\tilde{\Pi}_{6}^{nilp}(s)$
is given by
\[
\frac{4}{45}\ell^{6}(s)+\frac{5}{9}\ell^{4}(s)+q_{2}\ell^{2}(s)+q_{1}\frac{\zeta(5)}{(2\pi i)^{5}}\ell(s)+q_{0},
\]
where $q_{0},q_{1},q_{2}\in\QQ$.
\end{conjecture}
\appendix

\section{Some hypergeometric special values}

The vanishing cycle periods $\int_{\mu_{t}}\omega_{t}$ (for each
$d\geq1$) in $\S5.2$ are close cousins of the hypergeometric functions
${}_{d+1}F_{d}\left(\left.\begin{array}{c}
\frac{1}{2},\ldots,\frac{1}{2}\\
1,\ldots,1
\end{array}\right|t\right)=$
\[
\frac{2^{d-1}}{(2\pi i)^{d}}\int_{0\leq x_{d}\leq\cdots\leq x_{1}\leq1}\frac{dx_{1}\wedge\cdots\wedge dx_{d}}{\sqrt{\left(\prod_{i=1}^{d}x_{i}\right)(x_{1}-1)\left(\prod_{j=1}^{d-1}(x_{j+1}-x_{j})\right)(1-tx_{d})}},
\]
with the main discrepancy arising from the alternation between $x_{i}$
and $(1-x_{j})$ (rather than simply having $\prod x_{i}$) under
the radical in our setup. So it is not surprising that hypergeometric
special values such as \eqref{eqn gamma}-\eqref{eqn gammatilde}
appear in the coefficients of powers of $\log t$ in $\int_{\tau_{t}}\omega_{t}$. 

In this appendix, we shall explain how to derive expressions for these
constants in terms of Riemann zeta values. Writing (by abuse of notation)
$\zeta(1):=\log4$, we have the following for \eqref{eqn gamma}:\begin{equation}\label{e:g1eval}\gamma_1 = \zeta(1),\end{equation}\begin{equation}\gamma_2 = \zeta(2)-\frac{1}{2}\zeta(1)^2,\end{equation}\begin{equation}\label{e:g3eval}\gamma_3 = 2\zeta(3) - \zeta(2)\zeta(1) + \frac{1}{6}\zeta(1)^3,\end{equation}\begin{equation}\gamma_4=\frac{9}{4}\zeta(4) - 2\zeta(3)\zeta(1) + \frac{1}{2}\zeta(2)\zeta(1)^2 -\frac{1}{24}\zeta(1)^4,\end{equation}\begin{equation}\label{e:g5eval}\begin{matrix} \gamma_5=6\zeta(5)-\frac{9}{4}\zeta(4)\zeta(1)-2\zeta(3)\zeta(2)+\zeta(3)\zeta(1)^2-\frac{1}{6}\zeta(2)\zeta(1)^3 \\ +\frac{1}{120}\zeta(1)^5,\end{matrix}\end{equation}\begin{equation}\label{e:g6eval}\begin{matrix} \gamma_6 = \frac{79}{16}\zeta(6)-6\zeta(5)\zeta(1)+\frac{9}{8}\zeta(4)\zeta(1)^2-2\zeta(3)^2+2\zeta(3)\zeta(2)\zeta(1) \\ -\frac{1}{3}\zeta(3)\zeta(1)^3+\frac{1}{24}\zeta(2)\zeta(1)^4-\frac{1}{720}\zeta(1)^6.\end{matrix}\end{equation}We
also record some values of \eqref{eqn gammatilde}:%
\footnote{These suggest a delightful relation to the $\gamma_{n}$ which in
fact fails for $n\geq4$.%
} \begin{equation}\label{e:gteval}\tilde{\gamma}_0=1,\; \tilde{\gamma}_1=\zeta(1),\; \tilde{\gamma}_2=\zeta(2)+\frac{1}{2}\zeta(1)^2,\\ \tilde{\gamma}_3 = 2\zeta(3)+\zeta(2)\zeta(1)+\frac{1}{6}\zeta(1)^3.\end{equation}The
values $\gamma_{4},\gamma_{5},\gamma_{6},\tilde{\gamma}_{1},\tilde{\gamma}_{2},\tilde{\gamma}_{3}$
were computed using Mathematica \cite{MATH}, though the method used
below for $\gamma_{1},\gamma_{2},\gamma_{3}$ would also suffice.

We shall proceed by expressing the associated hypergeometric function
$$
     h_n(t) = \sum_{k=1}^{\infty}\, t^k 
              \begin{pmatrix} -1/2 \\ k \end{pmatrix} k^{-n}
            = (-t/2)\, {}_{n+2} F_{n+1} \left(\left.
             \begin{matrix} 1,  \cdots,  1,  \frac{3}{2}  \\
                      2,  \cdots,  2  \end{matrix}\right| -t\right) 
$$
in terms of polylogarithms for $n\leq 3$.  Since we are interested in 
$\gamma_n = h_n(-1)$, we introduce the auxiliary function
$f_n(u) = h_n(u^2-1)$. Manipulation of power series shows that
$$
     \frac{d h_n}{dt} = (1/t)h_{n-1}(t),\qquad h_n(0) = 0;
$$
and therefore
$$
     \frac{df_n}{du} = \frac{2u}{u^2 -1}f_{n-1}(u)
                     = \left(\frac{1}{u-1} + \frac{1}{u+1}\right)f_{n-1}(u).
$$
In particular, since $f_n(1) = h_n(0) = 0$, we have an iterated integral
formula
$$
     f_n(u) = \int_1^u\,  \frac{df_n}{dv}\,dv
            = \int_1^u\, \left(\frac{1}{v-1} + \frac{1}{v+1}\right)
                         f_{n-1}(v)\,dv
$$
which starts with
$$
     f_0(u) = h_0(u^2 - 1) = 1/u - 1
$$
by the binomial series.

\begin{lem} $f_1(u) = -2(\log(u+1)-\log(2))$.
\end{lem}
\begin{proof} By the above 
$$
       f_1(u) 
         = \int_1^u\, \frac{2v}{v^2-1}\left(1/v - 1\right)\,dv = -2\int_1^u \frac{dv}{1+v},
$$
which gives the result.
\end{proof}

It follows that $ \gamma_1 (= f_1(0))= \log(4)$ as desired.

\par At this point, we observe that the differential of $f_1(u)$ is 
$-2/(u+1)$, so we are really calculating a series of iterated integrals
on $\mathbb P^1 \setminus \{1,-1,\infty\}$.  We therefore make the change of 
variables $2w = v+1$ to obtain
$$
          f_n(2u-1) = \int_1^u\, \left(\frac{1}{w-1} + \frac{1}{w}\right)
                                 f_{n-1}(2w-1)\,dw
$$
Integration either by hand or Mathematica gives

\begin{lem}
$$
\aligned
           f_2(2w-1) = &2\Li_2(1-w) - \log^2(w) \\
           f_3(2w-1) = &2\Li_3(1-w) - 2\Li_3(w) + 2\zeta(3) \\
                       &+2\zeta(2)\log(w) -\log(1-w)\log^2(w)
                        -(1/3)\log^3(w)
\endaligned
$$
\end{lem}

\par To integrate $f_1(u)$ by hand to obtain $f_2(u)$, we observe that
$$
    f_1(u) = 2\sum_{\ell=1}^{\infty} \frac{(u-1)^{\ell} (-1)^{\ell}}{2^{\ell} \ell}
$$
at $u=1$.  This allows us to compute the integral of $f_1(u)/(u-1)$ in
terms of $\Li_2$.  The integral of $f_1(u)/(u+1)$ is elementary.
To determine $f_3(u)$ we must integrate both $f_2(v)/(v-1)$ and 
$f_2(v)/(v+1)$.  To this end, we use the following identity which allows
us to interchange $\Li_2(w)$ and $\Li_2(1-w)$:

\begin{lem}
$$
       \Li_2(w) + \Li_2(1-w) = \zeta(2) -\log(w)\log(1-w)
$$
and hence $2\Li_2(1/2) = \zeta(2) - \frac{1}{4}\zeta^2(1)$. 
\end{lem}
\begin{proof} Differentiate both sides with help of the formulas:
$$
        \frac{d}{dw}\Li_2(w) = -\frac{\log(1-w)}{w},\qquad
        \frac{d}{dw}\Li_2(1-w) = \frac{\log(w)}{1-w}
$$
The constant of integration $\zeta(2)$ is obtained by taking the limit at 
$w\to 1$.
\end{proof}

\par The two remaining integrals needed to calculate $f_3$ are given by:

\begin{lem} 
$$
\aligned
\int\, &\frac{\log^2(w)}{w-1} 
        = \log(1 - w)\log^2(w) + 2\log(w)\Li_2(w)- 2\Li_3(w) \\
\int\, &\frac{\log(w)\log(1-w)}{w}\,dw = \Li_3(w)-\Li_2(w)\log(w)
\endaligned
$$      
\end{lem}  
\begin{proof} Differentiate both sides of each equation using
$$
         \frac{d}{dw}\Li_3(w) = \frac{\Li_2(w)}{w},\qquad
         \frac{d}{dw}\Li_2(w) = -\frac{\log(1-w)}{w}.
$$
\end{proof}

\begin{rem} Based on empirical evidence gathered using Mathematica, 
it appears that
$$
         \gamma_d = \sum_P\, c_P\zeta(p_1)\cdots\zeta(p_r)
$$
where the sum runs over all partitions $P$ of $d$ into a sum of positive
integers, and the coefficients $c_P$ are determined as follows:
\begin{itemize}
\item The coefficient of $\zeta(d)$ in $\gamma_d$ is 
$$
            c_d = \frac{2^d - 2}{d},\qquad d>1.
$$
\item The coefficient of $\zeta(1)^d$ in $\gamma_d$ is $\frac{(-1)^{d-1}}{d!}$.
\item If $P = (p_1,\dots,p_r)$ is a partition of $d$ with all $p_j>1$ 
then
$$
        c_P = (-1)^{r-1} c_{p_1}\cdots c_{p_r}
$$
assuming all $p_j$'s are distinct. More generally, if $P$ contains
elements with multiplicity $m_1,\dots,m_k>1$ then
$$
        c_P = \frac{(-1)^{r-1}}{m_1!\cdots m_k!} c_{p_1}\cdots c_{p_r}
$$
For example $\zeta(a)\zeta(b)^2$ appears with coefficient 
$\frac{1}{2}c_ac_b^2$ while $\zeta(a)^3$ appears with coefficient 
$\frac{1}{6}c_a^3$.
\item The coefficient of $\zeta(p_1)\cdots\zeta(p_r)\zeta(1)^c$ is equal
to $\frac{(-1)^c}{c!}$ times the coefficient of 
$\zeta(p_1)\cdots\zeta(p_r)$.  This assumes all $p_j>1$.
\end{itemize}
For example, using $\gamma_2$, $\gamma_3$ we compute that 
$$
    \gamma_4 = \frac{7}{2}\zeta(4) - 2\zeta(3)\zeta(1) - \frac{1}{2}\zeta(2)^2
                  + \frac{1}{2}\zeta(2)\zeta(1)^2 - \frac{1}{24}\zeta(1)^4
$$
where $\frac{7}{2}\zeta(4)-\frac{1}{2}\zeta(2)^2 = \frac{9}{4}\zeta(4)$.  
Likewise, to calculate the coefficient of $\pi^6$ in $\gamma_6$ we consider
$$
      \frac{31}{3}\zeta(6) - \frac{7}{2}\zeta(4)\zeta(2) 
       + \frac{1}{6}\zeta(2)^3 = \frac{79}{16}\zeta(6)
$$
\end{rem}

\curraddr{\noun{${}$}\\
\noun{Department of Mathematics, Campus Box 1146}\\
\noun{Washington University in St. Louis}\\
\noun{St. Louis, MO} \noun{63130, USA}}

\email{\emph{${}$}\\
\emph{e-mail}: gffsjr@math.wustl.edu}

\curraddr{\noun{${}$}\\
\noun{Department of Mathematics, Campus Box 1146}\\
\noun{Washington University in St. Louis}\\
\noun{St. Louis, MO} \noun{63130, USA}}

\email{\emph{${}$}\\
\emph{e-mail}: matkerr@math.wustl.edu}

\curraddr{${}$\\
\noun{Mathematics Department, Mail stop 3368}\\
\noun{Texas A\&M University}\\
\noun{College Station, TX 77843, USA}}

\email{${}$\\
\emph{e-mail}: gpearl@math.msu.edu}
\end{document}